\documentclass{article}

\usepackage{fullpage,amsthm,amsmath,amssymb,amsbsy,amsfonts}

\title{\bf Generalized Poisson integral and sharp estimates for harmonic and biharmonic functions in the half-space}

\author{\sc{Gershon Kresin$^a\!\!$}
\thanks{e-mail: kresin@ariel.ac.il}$\;\;$
and \sc{Vladimir Maz'ya$^b$}
\thanks{e-mail: vladimir.mazya@liu.se}$\;\;$
\\ \\
{\it{$^a$Department of Mathematics, Ariel University, Ariel 40700, Israel}}\\
{\it{$^b$Department of Mathematical Sciences, University of Liverpool,
M$\&$O Building, Liverpool,}}\\ 
{\it{L69 3BX, UK; Department of Mathematics, Link\"oping University,SE-58183 Link\"oping,}}\\ 
{\it{{\hskip -145mm}Sweden}}
}
{ \date\ }

\numberwithin{equation}{section}
\newtheorem{lemma}{Lemma}
\newtheorem{theorem}{Theorem}
\newtheorem{proposition}[theorem]{Proposition}
\newtheorem{corollary}{Corollary}
\newenvironment{remark}{{\bf Remark}}

\newcommand{\bs}{\boldsymbol}

\begin{document}
\maketitle

{\bf Abstract.} A representation for the sharp coefficient in a pointwise estimate for the gradient 
of a generalized Poisson integral of a function $f$ on ${\mathbb R}^{n-1}$ is obtained under the assumption that
$f$ belongs to  $L^p$. It is assumed that the kernel of the integral depends on the parameters $\alpha$ and $\beta$.
The explicit formulas for the sharp coefficients are found for the cases $p=1$, $p=2$ and for some values of 
$\alpha , \beta$ in the case $p=\infty$. Conditions ensuring the validity of  some analogues of the Khavinson's 
conjecture for the  generalized Poisson integral are obtained. The sharp estimates are applied to harmonic and 
biharmonic functions in the half-space.
\\
\\
{\bf Keywords:} generalized Poisson integral, two-parametric kernel, sharp estimates, 
harmonic functions, biharmonic functions
\\
\\
{\bf 2010 Mathematics Subject Classification:} Primary: 31B10; Secondary: 31B05, 31B30
\\
\\
%%%%%%%%%%%%%%%%%%%%%%%%%%%%%%%%%%%%%%%%%%%%%%%%%%%%%%%%%%%%%%%%% 
\section{Background and main results}
%%%%%%%%%%%%%%%%%%%%%%%%%%%%%%%%%%%%%%%%%%%%%%%%%%%%%%%%%%%%%%%%%

In the paper \cite{KM2} (see also \cite{KM5}) a representation for the sharp coefficient
${\mathcal C}_p( x)$ in the inequality
$$
\left |\nabla v(x)\right |\leq {\mathcal C}_p( x)\big|\!\big |v \big |\!\big |_p
$$
was found, where $v$ is harmonic function in the half-space
${\mathbb R}^n _{+}=\big \lbrace x=(x', x_n): x'\in {\mathbb R}^{n-1}, x_n > 0\big \rbrace$,
represented by the Poisson integral with
boundary values in $L^p({\mathbb R}^{n-1})$,  $||\cdot ||_p$ is the norm in
$L^p({\mathbb R}^{n-1})$, $1\leq p \leq \infty$, $x \in {\mathbb R}^n _{+}$. It was shown that
$$
{\mathcal C}_p( x)=\frac{C_p}{x_{n}^{(n-1+p)/p}}
$$
and explicit formulas for $C_1$, $C_2$ and $C_\infty$ were found. Namely,
$$
C_1=\frac{2 (n-1)}{\omega_n}\;,\;\;\;\;\;\;\;\;\;C_2=\sqrt{\frac{(n-1)n}{2^n \omega_n}}\;,\;\;\;\;\;\;\;\;\;
C_\infty=\frac{4(n-1)^{(n-1)/2}\;\omega_{n-1}}{ n^{n/2}\;\omega_n}\;,
$$
where $\omega_n=2\pi^{n/2}/\Gamma(n/2)$ is the area of the unit sphere ${\mathbb S}^{n-1}$ in ${\mathbb R}^n$.

In \cite{KM2} it was shown that the sharp coefficients in pointwise estimates for the absolute value 
of the normal derivative and the modulus of the gradient of a harmonic function in the half-space 
coincide for the cases $p=1$, $p=2$ as well as for the case $p=\infty$. 

Similar results for the gradient and the radial derivative of a harmonic function in the multidimensional ball with 
boundary values from $L^p$ for $p=1, 2$ in \cite{KM3} were obtained.

We note that explicit sharp 
coefficients in the inequality for the first derivative of analytic function in the half-plane %%%and the disk 
with boundary values of the real-part from $L^p$ in \cite{KM4} were found.

The subjects of papers \cite{KM2,KM3} is closely connected with D. Khavinson problem \cite{KHAV} and conjecture 
(see \cite{KM5}, Ch. 6) for harmonic functions in a ball. The Khavinson's problem is to find the sharp coefficient in the inequality
\begin{equation} \label{EI_1K}
|\nabla  v(x)| \leq {\mathcal K}(x)\sup _{|y|<1}|v(y)|,
\end{equation}
where $v$ is a bounded harmonic function in the ball $B=\{x \in {\mathbb R}^3: |x|<1 \}$. The Khavinson's conjecture 
is that the sharp coefficient ${\mathcal K}(x)$ in (\ref{EI_1K}) and the sharp coefficient $K(x)$ 
in the inequality
$$
\left |\frac{\partial v(x)}{\partial |x|} \right |\leq K(x)\sup _{|y|<1}|v(y)|
$$
coincide for any $x \in B$.

Thus, the $L^1, L^2$-analogues of Khavinson's problem were solved in \cite{{KM2},{KM3}} 
for harmonic functions in the multidimensional half-space and the ball. Also, the $L^\infty$-analogue of Khavinson's problem
for harmonic functions in the multidimensional half-space in \cite{KM2} was solved. A similar remark concerns to the positive 
answers on aforementioned analogues of the Khavinson's conjecture for harmonic functions in the half-space and the ball. 

\smallskip
In this paper we treat a generalization of some problems considered in our work \cite{KM2}.
Here we consider the generalized Poisson integral in the half-space
\begin{equation} \label{EI_1}
u(x)=k \int _{{\mathbb R}^{n-1}}\left (\frac{x_{n}^\alpha}{|y-x|}\right )^\beta\;f(y')dy'
\end{equation}
with two parameters, $\alpha\geq 0$ and $\beta >(n-1)(p-1)/p$, where $k$ is a constant, $n>2$, $f\in L^p({\mathbb R}^{n-1})$, 
$1\leq  p\leq \infty $, $y=(y', 0)$, $y' \in {\mathbb R}^{n-1}$.  

In the case $\alpha=1/n, \beta=n, k=2/\omega_{n}$ integral (\ref{EI_1}) coincides 
with the Poisson integral for harmonic functions in the half-space.
If $k=2/((n-2)\omega_{n})$, $\alpha=0$ and $\beta=n-2$, then integral (\ref{EI_1}) gives solution 
of the Neumann problem for the half-space. Solution of the first boundary value problem for the biharmonic equation 
in the half-space is represented as the sum of two integrals (\ref{EI_1}) with $\alpha=3/(n+2), \beta=n+2$ and 
$\alpha=2/n, \beta=n$, accordingly. Integral (\ref{EI_1}) with $\alpha=0$, $\beta\in (0, n-1)$ with appropriate choice of $k$ 
can be considered as a continuation on ${\mathbb R}^n$ of the Riesz potential in ${\mathbb R}^{n-1}$. 

In the present paper we arrive at conditions for which some analogues of Khavinson's conjecture for the generalized Poisson integral 
in the half-space are valid.

\smallskip
In Section 2 we obtain a representation for the sharp coefficient ${\mathcal C}_{\alpha, \beta, p}(x)$ in the inequality
\begin{equation} \label{EI_2}
\left |\nabla u(x)\right |\leq {\mathcal C}_{\alpha, \beta, p}( x)\big|\!\big |f \big |\!\big |_p\;,
\end{equation}
where 
$$
{\mathcal C}_{\alpha, \beta, p}(x)=\frac{C_{\alpha, \beta, p}}{x_{n}^{2-n+\beta(1-\alpha)+((n-1)/p)}}\;.
$$
The constant $C_{\alpha, \beta, p}$ in this section is characterized in terms of an extremal problem on the sphere ${\mathbb S}^{n-1}$.

\smallskip
Analyzing this extremal problem for the case $p=1$, in Section 3 we derive the explicit formula for $C_{\alpha, \beta, 1}$
with $\beta>0$. It is shown that
$C_{\alpha, \beta, 1}=|k|\beta|1-\alpha|$ if $\alpha$ satisfies the condition  
$$
0\leq\alpha\leq\frac{\sqrt{1+\beta}}{\sqrt{1+\beta}+1}\;\;\;{\rm or}\;\;\;\alpha\geq\frac{\sqrt{1+\beta}}{\sqrt{1+\beta}-1}\;.
$$  
For these values of $\alpha$, the constant $C_{\alpha, \beta, 1}=|k|\beta|1-\alpha|$ is sharp also in the weaker inequality 
obtained from $(\ref{EI_2})$ with $p=1$ by replacing $\nabla u$ 
by $\partial u/\partial x_n$. Also, it is shown that
$$
C_{\alpha, \beta, 1}=|k|\beta\left ( \frac{\beta }{2\alpha-1} \right )^{\beta/2}\left ( \frac{\alpha^2}{1+\beta } \right )^{(\beta+2)/2} 
$$
if $\alpha$ satisfies the condition
$$
\frac{\sqrt{1+\beta}}{\sqrt{1+\beta}+1} < \alpha< \frac{\sqrt{1+\beta}}{\sqrt{1+\beta}-1}\;.
$$

\smallskip
In Section 4 we consider the case $\alpha=0$ in (\ref{EI_1}). Solving the extremal problem on ${\mathbb S}^{n-1}$ described in Section 2, 
we arrive at the explicit formula for the sharp coefficient ${\mathcal C}_{0, \beta, p}(x)$ in inequality (\ref{EI_2}) with $\alpha=0$. 
In particular, we obtain the sharp inequality 
\begin{equation} \label{EI_4}
\left |\nabla u (x) \right |\leq \frac{ C_{0, \beta, p}}{x_{n}^{2-n+\beta+((n-1)/p)}}\;
\big|\!\big |f \big |\!\big |_p
\end{equation}
for $\beta> n-1$ and $p\in [1, \infty ]$, where $C_{0, \beta, 1}=|k|\beta$ and
\begin{equation} \label{EI_5}
C_{0, \beta, p}=|k|\beta\left \{\frac{\pi^{\frac{n-1}{2}}\Gamma\left ( \frac{(\beta -n +3)p+n-1}{2(p-1)} \right )}
{\Gamma\left ( \frac{(\beta +2)p}{2(p-1)} \right )} \right \}^{\frac{p}{p-1}}
\end{equation}
for $p>1$. The constant (\ref{EI_5}) is sharp also in the weaker inequality 
obtained from $(\ref{EI_4})$ by replacing $\nabla u $ by $\partial u/\partial x_n$. 

\smallskip
In Section 5 we reduce the extremal problem on the sphere ${\mathbb S}^{n-1}$ from Section 2 to that of finding of the supremum
of a certain double integral, depending on a scalar parameter.

\smallskip
Using the representation for $C_{\alpha, \beta, p}$ as the supremum of the double integral with a scalar parameter from Section 5, in Section 6 we 
consider the case $p=2$. Here we obtain results similar to those of Section 3.

\smallskip
In Section 7 we deal with the case $p=\infty$ in (\ref{EI_2}). First, we show that for any $\beta>n-1$ there exists $\alpha_n(\beta)>1$
such that for $\alpha\geq \alpha_n(\beta)$ the equality holds
$$
C_{\alpha, \beta, \infty}=|k|\frac{\pi^{(n-1)/2}\;\Gamma\left ( \frac{\beta -n+1}{2}\right )}{
\Gamma\left ( \frac{\beta}{2}\right )}\big ( (\alpha-1)\beta + n-1 \big ).
$$
The number $\alpha_n(\beta)$ is a root of a transcendental equation. %%%We calculate it for some values of $n$ and $\beta$.
For instance, $\alpha_3(2.5)\approx 1.2865, \;\alpha_3(3)\approx 1.4101, \;\alpha_3(3.5)\approx 1.4788$. 
Second, we consider the case $\alpha=1$ separately and show that
$$
C_{1, \beta, \infty}=|k|\frac{\pi^{(n-1)/2}(n-1)\Gamma\left (\frac{\beta-n+1}{2} \right )}{\Gamma\left (\frac{\beta}{2} \right )}
$$
for $\beta \in (n-1, n]$. In each of two assertions of Section 7 we show that absolute value of the derivative of $u$
with respect to the normal to the boundary of the half-space at any point
$x \in {\mathbb R}^n_+$ has the same supremum as $|\nabla u(x)|$. 

\smallskip
In Section 8 we concretize the results of Sections 3, 4, 6, 7 to obtain sharp estimates for the gradient of $x_n^{n\kappa -1} v(x)$,
where $\kappa\geq 0$ and $v$ is a harmonic function in ${\mathbb R}^n_+$ which can be represented as the Poisson integral with boundary values from $L^p$. Here we give the explicit formulas for the sharp constant $C_{\kappa, n, p}$ in the estimate 
\begin{equation} \label{EI_A}
\left |\nabla \big ( x_n^{n\kappa - 1}v(x)\big )\right |\leq C_{\kappa, n, p}x_n^{n\kappa-2-(n-1)/p}\big|\!\big |v \big |\!\big |_p
\end{equation}
with some values of $\kappa$ and $p$. For instance, in the case $\kappa=0 $ we derive the inequality
\begin{equation} \label{EI_8A}
\left |\nabla \left \{ \frac{v (x)}{x_n} \right \}\right |
\leq C_{0, n, p} x_{n}^{-2-(n-1)/p}\;\big|\!\big | v \big |\!\big |_p
\end{equation}
with the sharp constant
$$
C_{0, n, p}=\frac{2 n}{\omega_n}\left \{\frac{\pi^{\frac{n-1}{2}}\Gamma\left ( \frac{3p+n-1}{2(p-1)} \right )}
{\Gamma\left ( \frac{(n +2)p}{2(p-1)} \right )} \right \}^{\frac{p}{p-1}}
$$
for $1< p< \infty$. For the cases $p=1, 2, \infty$ inequality (\ref{EI_8A}) becomes
$$
\left |\nabla \left \{ \frac{v (x)}{x_n} \right \}\right |
\leq \frac{2n}{\omega_n} \frac{1}{x_{n}^{n+1}}\;\big|\!\big | v \big |\!\big |_1\;,\;\;\;\;\;\;
\left |\nabla \left \{ \frac{v (x)}{x_n} \right \}\right |
\leq \sqrt{\frac{n(n+3)}{2^n\omega_n}} \frac{1}{x_{n}^{(n+3)/2}}\;\big|\!\big | v \big |\!\big |_2\;,\;\;\;\;\;\;
\left |\nabla \left \{ \frac{v (x)}{x_n} \right \}\right |
\leq \frac{1}{x_n^2} \;\big|\!\big | v \big |\!\big |_\infty\;,
$$
accordingly. We note, that the constants in inequality (\ref{EI_8A}) %%%and (\ref{EI_8B}) 
remain sharp also in the weaker inequalities obtained by replacing $\nabla $ by $\partial /\partial x_n$.

We mention one more group of  inequalities for harmonic functions with the sharp coefficients obtained for the case $\kappa=1$:
$$
\left |\nabla \big ( x_n^{n - 1}v(x)\big )\right |\leq \frac{2(n-2)}{n\omega_n}\left \{ \frac{(n-1)^2}{(n-2)(n+1)} \right \}^{(n+2)/2}\; \frac{1}{x_n} \big|\!\big | v \big |\!\big |_1\;,
$$ 
$$
\left |\nabla \big ( x_n^{n - 1}v(x)\big )\right |\leq \left \{ \frac{n(n-1)}{2^{n}\omega_n }\right \}^{1/2} 
x_n^{(n-3)/2} \big|\!\big | v \big |\!\big |_2\;,\;\;\;\;\;\;\;
\left |\nabla \big ( x_n^{n - 1}v(x)\big )\right |\leq (n-1) x_n^{n-2} \big|\!\big | v \big |\!\big |_\infty\;.
$$

\smallskip
Concluding of Section 8, we present the sharp estimate 
\begin{equation} \label{EI_8w}
\left |\nabla \big ( x_n^{n\kappa - 2}w_0(x)\big )\right |\leq C_{\kappa, n, p}x_n^{n\kappa-2-(n-1)/p}\left |\!\left |
\frac{\partial w_0}{\partial x_n} \right |\!\right |_p\;,
\end{equation}
where $\kappa\geq 0$ and $w_0$ is a biharmonic function in ${\mathbb R}^n_+$ with the boundary values 
$$
w_0\big |_{x_n=0}=0\;,\;\;\;\;\;\;\;\;\frac{\partial w_0}{\partial x_n}\Big |_{x_n=0} \in L^p({\mathbb R}^{n-1})\;.
$$
The sharp constant $C_{\kappa, n, p}$ in inequality (\ref{EI_8w}) is the same as in (\ref{EI_A}). 
For example, in the case $\kappa=0, p=\infty$, inequality (\ref{EI_8w}) takes the form
$$
\left |\nabla \left \{ \frac{w_0 (x)}{x_n^2} \right \}\right |
\leq \frac{1}{x_n^2} \;\left |\!\left |
\frac{\partial w_0}{\partial x_n} \right |\!\right |_\infty\;.
$$

%%%%%%%%%%%%%%%%%%%%%%%%%%%%%%%%%%%%%%%%%%%%%%%%%%%%%%%%%%%%%%%%%
\section{Representation for the sharp constant in inequality for 
the gradient in terms of an extremal problem on the unit sphere}
%%%%%%%%%%%%%%%%%%%%%%%%%%%%%%%%%%%%%%%%%%%%%%%%%%%%%%%%%%%%%%%%%

We introduce some notation used henceforth.
Let ${\mathbb R}^{n} _{+}=\big \lbrace x=(x', x_{n}): x'=(x_1,\dots,x_{n-1})
\in {\mathbb R}^{n-1}, x_{n} > 0\big \rbrace$, ${\mathbb S}^{n-1}=
\{ x \in {\mathbb R}^{n}: |x|=1 \}$,
${\mathbb S}^{n-1}_+= \{ x \in {\mathbb R}^{n}: |x|=1,\; x_{n} >0 \}$ and
${\mathbb S}^{n-1}_{-}= \{ x \in {\mathbb R}^{n}: |x|=1,\; x_{n}<0 \}$.
Let $\bs e_\sigma $ stand for the
$n$-dimensional unit vector joining the origin to a point $\sigma $
on the sphere ${\mathbb S}^{n-1}$. 
As before, by $\omega_n=2\pi ^{n/2}/\Gamma(n/2)$ we 
denote the area of the unit sphere in ${\mathbb R}^n$. Let $\bs e_n$ be the unit vector of the $n$-th coordinate axis.

By $||\cdot ||_p$ we denote the norm in the space $L^p({\mathbb R}^{n-1})$, that is
$$
|| f||_p=\left \{\int _ {{\mathbb R}^{n-1}} |f(x')|^p \;dx' \right \}^{1/p},
$$
if $1\leq  p< \infty $, and $|| f||_\infty =\mbox{ess}\;\sup \{ | f(x') |:
x' \in {\mathbb R}^{n-1} \}$.

Let the function $u$ in ${\mathbb R}^{n}_{+}$ be represented as the generalized Poisson integral
\begin{equation} \label{EH_1}
u(x)=k \int _{{\mathbb R}^{n-1}}\left (\frac{x_{n}^\alpha}{|y-x|}\right )^\beta\;f(y')dy'
\end{equation}
with parameters $\alpha\geq 0$ and
\begin{equation} \label{EQU_1}
\beta >(n-1)(p-1)/p\;, 
\end{equation}
where $k$ is a constant, $f\in L^p({\mathbb R}^{n-1})$, $1\leq  p\leq \infty $, $y=(y', 0)$, $y' \in {\mathbb R}^{n-1}$.

\smallskip
Now, we find a representation for the best coefficient ${\mathcal C}_p( x; \bs z)$
in the inequality for the absolute value of derivative of $u(x)$ 
in an arbitrary direction $\bs z \in {\mathbb S}^{n-1}$, $x\in {\mathbb R}^{n}_{+}$.
In particular, we obtain a formula for the sharp coefficient in a similar inequality for the
modulus of the gradient.

\begin{proposition} \label{P_1}
Let $x $ be  an arbitrary point in ${\mathbb R}^{n} _+$ and let $\bs z\in {\mathbb S}^{n-1}$. 
The sharp coefficient ${\mathcal C}_{\alpha, \beta, p} (x; \bs z)$ in the inequality
\begin{equation} \label{EH_2A1}
|\left ( \nabla u(x), \bs z \right ) |\leq {\mathcal C}_{\alpha, \beta, p}( x; \bs z)
\big|\!\big |f \big |\!\big |_p
\end{equation}
is given by
\begin{equation} \label{EH_2A}
{\mathcal C}_{\alpha, \beta, p}( x; \bs z)= \frac{C_{\alpha, \beta, p}(\bs z)}{x_{n}^{2-n+\beta(1-\alpha)+((n-1)/p)}}\;,
\end{equation}
where
\begin{equation} \label{EH_2B}
C_{\alpha, \beta, 1}(\bs z)=|k|\beta\sup _{\sigma \in {\mathbb S}^{n-1}_+}
\big |\big (\alpha\bs e_{n} -(\bs e_{\sigma}, \bs e_{n})\bs e_{\sigma},\; \bs z \big )\big |
\big (\bs e_{\sigma}, \bs e_{n} \big )^\beta ,
\end{equation}
\begin{equation} \label{EH_3DA}
C_{\alpha, \beta, p}(\bs z)\! =\!|k|\beta
\left \{ \int _ {{\mathbb S}^{n-1}_+ }\!\!
\big |\!\big (\alpha\bs e_{n} \!-\!(\bs e_{\sigma}, \bs e_{n})\bs e_{\sigma}, \bs z \big )\!\big |^{\frac{p}{p\!-\!1}}
\big (\bs e_{\sigma}, \bs e_{n} \big )^{\frac{(\beta - n)p+n}{p-1}}d\sigma \!\right \}^{\frac{p\!-\!1}{p}}
\end{equation}
for $1<p<\infty $,
and
\begin{equation} \label{EH_3H}
C_{\alpha, \beta, \infty}(\bs z)=|k|\beta 
\int _ {{\mathbb S}^{n-1}_+ }
\big |\big (\alpha\bs e_{n} - (\bs e_{\sigma}, \bs e_{n})\bs e_{\sigma},\; 
\bs z \big )\big |
\big (\bs e_{\sigma}, \bs e_{n} \big )^{\beta-n}\;d\sigma .
\end{equation}

In particular, the sharp coefficient ${\mathcal C}_{\alpha, \beta, p} (x)$ in the inequality
\begin{equation} \label{EH_2A2}
\left |\nabla u(x)\right |\leq {\mathcal C}_{\alpha, \beta, p}( x)\big|\!\big |f \big |\!\big |_p
\end{equation}
is given by
\begin{equation} \label{EH_3J}
{\mathcal C}_{\alpha, \beta, p}(x)=\frac{C_{\alpha, \beta, p} }{x_{n}^{2-n+\beta(1-\alpha)+((n-1)/p)}}\;,
\end{equation}
where
\begin{equation} \label{EH_4J}
C_{\alpha, \beta, p}=\sup _{|\bs z|=1}C_{\alpha, \beta, p}(\bs z).
\end{equation}
\end{proposition}
\begin{proof} Let $x=(x', x_{n})$ be a fixed point in ${\mathbb R}^{n}_+$. 
The representation (\ref{EH_1}) implies
$$
\frac{\partial u}{\partial x_i}=k\int _{{\mathbb R}^{n-1}}
\left [ \frac{\alpha\beta\delta_{n i} x_{n}^{\alpha\beta-1}}{|y-x|^\beta} + \frac{\beta 
x_{n} ^{\alpha\beta}(y_i-x_i)}{|y-x|^{\beta+2}}  \right ]f(y')dy',
$$
that is
\begin{eqnarray*}
\nabla u(x)&=&k \beta x_{n}^{\alpha\beta-1}\int _{{\mathbb R}^{n-1}}
\left [\; \frac{\alpha\bs e_{n}}{|y-x|^{\beta}} + \frac{ x_{n} (y-x)}{|y-x|^{\beta+2}}\; \right ]f(y')dy'\\
& &\\
&=&k \beta x_{n}^{\alpha\beta-1}\int _{{\mathbb R}^{n-1}}
\frac{\alpha\bs e_{n}  - (\bs e_{xy}, \bs e_{n}) \bs e_{xy}}{|y-x|^{\beta}}\; f(y')dy',
\end{eqnarray*}
where $\bs e_{xy}=(y-x)|y-x|^{-1}$.
For any $\bs z \in {\mathbb S}^{n-1}$,
$$
(\nabla u(x), \bs z)=
k \beta x_{n}^{\alpha\beta-1}\int _{{\mathbb R}^{n-1}}
\frac{(\alpha\bs e_{n}  - (\bs e_{xy}, \bs e_{n})\bs e_{xy},\; \bs z)}{|y-x|^{\beta}}\; f(y')dy'.
$$
Hence,
\begin{equation} \label{EH_7AD}
{\mathcal C}_{\alpha, \beta, 1}( x; \bs z)=|k| \beta x_{n}^{\alpha\beta-1}
\sup _{y \in {\partial\mathbb R}^{n-1}}
\frac{|(\alpha\bs e_{n}  - (\bs e_{xy}, \bs e_{n})\bs e_{xy},\; \bs z)|}{|y-x|^{\beta}}\;,
\end{equation}
and
\begin{equation} \label{EH_6A}
{\mathcal C}_{\alpha, \beta, p}( x; \bs z)=|k| \beta x_{n}^{\alpha\beta-1}
\left \{\int _{{\mathbb R}^{n-1}}
\frac{\big |\big (\alpha\bs e_{n}  - (\bs e_{xy}, \bs e_{n})\bs e_{xy}, 
\bs z \big )\big |^q }{|y-x|^{\beta q}}\;dy' \right \}^{1/q}
\end{equation}
for $1<p \leq \infty $, where $p^{-1}+q^{-1}=1$.

Taking into account the equality
\begin{equation} \label{EH_COS}
\frac{x_{n}}{|y-x|}=(\bs e_{xy}, -\bs e_{n}),
\end{equation}
by (\ref{EH_7AD}) we obtain
\begin{eqnarray*}
{\mathcal C}_{\alpha, \beta, 1}( x; \bs z)
&=&|k| \beta x_{n}^{\alpha\beta-1}\sup _{y \in \partial{\mathbb R}^{n-1}}
\frac{|(\alpha\bs e_{n}  - (\bs e_{xy}, \bs e_{n})\bs e_{xy},\; \bs z)|}{x_{n}^{\beta}}
\left (\frac{x_{n}}{|y-x|}\right )^{\beta} \\
& &\\
&=&\frac{|k| \beta}{x_{n}^{1+\beta(1-\alpha)}}\sup _{\sigma \in {\mathbb S}^{n-1}_-}
\big |\big (\alpha\bs e_{n} -(\bs e_{\sigma}, \bs e_{n})\bs e_{\sigma},\; \bs z \big )\big |\big (\bs e_{\sigma}, -\bs e_{n} 
\big )^{\beta}.
\end{eqnarray*}
Replacing here $\bs e_\sigma$ by $-\bs e_\sigma$, we arrive at (\ref{EH_2A})
for $p=1$  with the sharp constant (\ref{EH_2B}).

Let $1<p \leq \infty $. Using (\ref{EH_COS}) and the equality
$$
\frac{1}{|y-x|^{\beta q}}=\frac{1}{x_{n}^{\beta q-n+1}}\left (\frac{x_{n}}{|y-x|} \right )^{\beta q-n}
\frac{x_n}{|y-x|^{n}}\;,
$$
and replacing  $q$ by $p/(p-1)$ in (\ref{EH_6A}), we conclude that (\ref{EH_2A}) holds
with the sharp constant
$$
C_{\alpha, \beta, p}(\bs z)=|k|\beta
\left \{ \int _ {{\mathbb S}^{n-1}_- }
\big |\big (\alpha\bs e_{n} -(\bs e_{\sigma}, \bs e_{n})\bs e_{\sigma},\; \bs z \big )\big |^{\frac{p}{p-1}}
\big (\bs e_{\sigma}, -\bs e_{n} \big )^{\frac{(\beta -n )p+n}{p-1}}\;d\sigma \right \}^{\frac{p-1}{p}},
$$
where ${\mathbb S}^{n-1}_- =\{ \sigma \in {\mathbb S}^{n-1}: (\bs e_\sigma , \bs e_{n})<0 \}$.
Replacing here $\bs e_\sigma$ by $-\bs e_\sigma$, 
we arrive at (\ref{EH_3DA}) for $1<p<\infty$ and at (\ref{EH_3H}) for $p=\infty$.

Estimate (\ref{EH_2A2}) with the sharp coefficient (\ref{EH_3J}), where the constant $C_{\alpha, \beta, p}$ is given by (\ref{EH_4J}),
is an immediate consequence of (\ref{EH_2A1}) and (\ref{EH_2A}).
\end{proof}

\begin{remark}. 
Formula (\ref{EH_3DA}) for the sharp constant $C_{\alpha, \beta, p}(\bs z)$ in (\ref{EH_2A}), $1<p< \infty$,
can be written with the integral over the whole sphere ${\mathbb S}^{n-1}$
in ${\mathbb R}^{n}$,
\begin{equation} \label{EH_3HW1}
C_{\alpha, \beta, p}(\bs z)=\frac{|k|\beta}{2^{(p-1)/p}}\;
\left \{ \int _ {{\mathbb S}^{n-1}}\!\!
\big |\!\big (\alpha\bs e_{n} \!-\!(\bs e_{\sigma}, \bs e_{n})\bs e_{\sigma}, \bs z \big )\!\big |^{\frac{p}{p\!-\!1}}
\big |\big (\bs e_{\sigma}, \bs e_{n} \big )\big |^{\frac{(\beta - n)p+n}{p-1}}d\sigma \!\right \}^{\frac{p\!-\!1}{p}}\;.
\end{equation}
A similar remark relates (\ref{EH_2B}):
\begin{equation} \label{EH_3HW3}
C_{\alpha, \beta, 1}(\bs z)=|k|\beta\sup _{\sigma \in {\mathbb S}^{n-1}}
\big |\big (\alpha\bs e_{n} -(\bs e_{\sigma}, \bs e_{n})\bs e_{\sigma},\; 
\bs z \big )\big |\big |\big (\bs e_{\sigma}, \bs e_{n} \big )\big |^\beta\;,
\end{equation}
as well as formula (\ref{EH_3H}):
$$
C_{\alpha, \beta, \infty}(\bs z)=\frac{|k|\beta}{2}
\int _ {{\mathbb S}^{n-1}}
\big |\big (\alpha\bs e_{n} - (\bs e_{\sigma}, \bs e_{n})\bs e_{\sigma},\; \bs z \big )\big |
\big |\big (\bs e_{\sigma}, \bs e_{n} \big )\big |^{\beta-n}\;d\sigma\;. 
$$
\end{remark}

%%%%%%%%%%%%%%%%%%%%%%%%%%%%%%%%%%%%%%%%%%%%%%%%%%%%%%%%%%%%%%%%%
\section{The case $p=1$} 
%%%%%%%%%%%%%%%%%%%%%%%%%%%%%%%%%%%%%%%%%%%%%%%%%%%%%%%%%%%%%%%%%

In the next assertion we obtain the explicit formula for the sharp constant $C_{\alpha, \beta, 1}$.

\setcounter{theorem}{0}
\begin{theorem} \label{C_1_1} Let $f \in L^1({\mathbb R}^{n-1})$,
and let $x $ be  an arbitrary point in ${\mathbb R}^{n} _+$. 
The sharp coefficient ${\mathcal C}_{\alpha, \beta, 1} (x)$ in the inequality
\begin{equation} \label{Eq_C_1.1}
\left |\nabla u(x)\right |\leq {\mathcal C}_{\alpha, \beta, 1}( x)\big|\!\big |f \big |\!\big |_1
\end{equation}
is given by
\begin{equation} \label{Eq_C_1.2}
{\mathcal C}_{\alpha, \beta, 1}(x)=\frac{C_{\alpha, \beta, 1}}{x_{n}^{1+\beta(1-\alpha)}}\;,
\end{equation}
where
\begin{equation} \label{EH_2BAA}
C_{\alpha, \beta, 1}=|k|\beta|1-\alpha|
\end{equation}
if 
\begin{equation} \label{EH_3AA}
0\leq\alpha\leq\frac{\sqrt{1+\beta}}{\sqrt{1+\beta}+1}\;\;\;{\rm or}\;\;\;\alpha\geq\frac{\sqrt{1+\beta}}{\sqrt{1+\beta}-1}\;,
\end{equation}
and
\begin{equation} \label{EH_4AA}
C_{\alpha, \beta, 1}=|k|\beta\left ( \frac{\beta }{2\alpha-1} \right )^{\beta/2}\left ( \frac{\alpha^2}{1+\beta } \right )^{(\beta+2)/2} 
\end{equation}
if
\begin{equation} \label{EH_5AA}
\frac{\sqrt{1+\beta}}{\sqrt{1+\beta}+1} < \alpha< \frac{\sqrt{1+\beta}}{\sqrt{1+\beta}-1}\;.
\end{equation}

If $\alpha$ satisfies condition $(\ref{EH_3AA})$, then the coefficient ${\mathcal C}_{\alpha, \beta, 1}( x)$ 
is sharp also in the weaker inequality obtained from $(\ref{Eq_C_1.1})$ by 
replacing $\nabla u$ by $\partial u /\partial x_n$.
\end{theorem}
\begin{proof} The equality (\ref{Eq_C_1.2}) for the sharp coefficient ${\mathcal C}_{\alpha, \beta, 1}(x)$ in (\ref{Eq_C_1.1}) 
was proved in Proposition \ref{P_1}. Using (\ref{EH_2B}), (\ref{EH_4J}) and the permutability of two suprema, we find
\begin{eqnarray} \label{EH_2BP}
C_{\alpha, \beta, 1}=|k|\beta\sup_{|\bs z|=1}\sup _{\sigma \in {\mathbb S}^{n-1}_+ }
\big |\big (\alpha\bs e_n -(\bs e_{\sigma}, \bs e_n)\bs e_{\sigma},\; \bs z \big )\big |\big (\bs e_{\sigma}, \bs e_n \big )^\beta 
=|k|\beta\sup _{\sigma \in {\mathbb S}^{n-1}_+ }
\big |\alpha\bs e_n -(\bs e_{\sigma}, \bs e_n)\bs e_{\sigma}\big |\big (\bs e_{\sigma}, \bs e_n \big )^\beta  \;.
\end{eqnarray}
Taking into account the equality
\begin{eqnarray*}
\big |\alpha\bs e_n -n(\bs e_{\sigma}, \bs e_n)\bs e_{\sigma}\big |
=\Big (\alpha\bs e_n -(\bs e_{\sigma}, \bs e_n)\bs e_{\sigma}, 
\;\alpha\bs e_n -(\bs e_{\sigma}, \bs e_n)\bs e_{\sigma} \Big )^{1/2}
=\Big (\alpha^2+(1-2\alpha)(\bs e_{\sigma},\; \bs e_n)^2 \Big )^{1/2},
\end{eqnarray*}
and using (\ref{EH_2BP}), we arrive at the representation
\begin{equation} \label{EH_2BPA}
C_{\alpha, \beta, 1}=|k|\beta\sup _{\sigma \in {\mathbb S}^{n-1}_+ }\Big (\alpha^2+(1-2\alpha)(\bs e_{\sigma},\; \bs e_n)^2 \Big )^{1/2}
\big (\bs e_{\sigma}, \bs e_n \big )^\beta \;.
\end{equation}

We denote $t=\big (\bs e_{\sigma}, \bs e_n \big )$. Let us introduce the function
\begin{equation} \label{EH_1BPAB}
f(t)=\big ( \alpha^2+(1-2\alpha)t^2\big )^{1/2}t^{\beta},
\end{equation}
where $t\in [0, 1]$, $\alpha\geq 0$ and $\beta>0$. By (\ref{EH_2BPA}),
\begin{equation} \label{EH_2BPAB}
C_{\alpha, \beta, 1}=|k|\beta\max _{0\leq t\leq 1}f(t)\;.
\end{equation}
Taking into account that $f(t)>0$ for $t\in (0, 1)$ and any $\alpha\geq 0$, $\beta>0$, we can consider the 
function $F(t)=f^2(t)$ on the interval $t\in (0, 1)$ instead of $f(t)$. We have
\begin{equation} \label{EH_2BPAC}
F'(t)=2\Big (\alpha^2\beta +(1-2\alpha)(1+\beta)t^2 \Big )t^{2\beta-1}\;.
\end{equation}

If $0\leq \alpha\leq1/2$, then $F'(t)>0$ for $t\in(0, 1)$. If $\alpha>1/2$, then the positive root of the equation
$F'(t)=0$ is
\begin{equation} \label{EH_2BPAH}
t_1=\sqrt{\frac{\alpha^2\beta}{(2\alpha -1)(1+\beta)}}\;.
\end{equation}
Herewith, if
\begin{equation} \label{EH_2BPAD}
\frac{\alpha^2\beta}{(2\alpha -1)(1+\beta)}\geq 1\;,
\end{equation}
then $t_1\notin (0, 1)$. Solving inequality (\ref{EH_2BPAD}) with respect to $\alpha$, we obtain intervals for which (\ref{EH_2BPAD})
holds:
$$
\alpha \leq \alpha_{_1}=\frac{\sqrt{1+\beta}}{\sqrt{1+\beta}+1}\;\;\;{\rm and}\;\;\;\alpha \geq \alpha_{_2}=\frac{\sqrt{1+\beta}}{\sqrt{1+\beta}-1}\;.
$$
Hence, $F'(t)=(f^2(t))'>0$ for $t\in (0, 1)$ if $\alpha\leq \alpha_{_1}$ or $\alpha\geq \alpha_2$. This, by (\ref{EH_1BPAB}) and (\ref{EH_2BPAB}), 
proves the equality (\ref{EH_2BAA}) for (\ref{EH_3AA}).

Furthermore, by (\ref{EH_2B}),
$$
C_{\alpha, \beta, 1}(\bs e_n)=|k|\beta\sup _{\sigma \in {\mathbb S}^{n-1}_+ }
\big |\alpha -(\bs e_{\sigma}, \bs e_n)^2 |\big (\bs e_{\sigma}, \bs e_n \big )^\beta \geq
|k|\beta|1-\alpha|\;.
$$
Hence, by $C_{\alpha, \beta, 1} \geq C_{\alpha, \beta, 1}(\bs e_n)$ and by (\ref{EH_2BAA}) we obtain $C_{\alpha, \beta, 1}=
C_{\alpha, \beta, 1}(\bs e_n)$, which completes the proof for the case $\alpha \leq \alpha_{_1}$ as for as for the case 
$\alpha \geq \alpha_2$.

Now, we consider the case $t_1<1$, that is
$$
\frac{\alpha^2\beta}{(2\alpha -1)(1+\beta)}< 1\;.
$$ 
The last inequality holds for $\alpha\in (\alpha_{_1}, \alpha_{_2})$. Differentiating (\ref{EH_2BPAC}), we obtain
$$
F''(t)=2\Big (\alpha^2\beta +(1-2\alpha)(1+\beta)(1+2\beta)t^2 \Big )t^{2(\beta-1)}\;.
$$
After calculations, we have
$$
F''(t_1)=-4\alpha^2\beta^2\left ( \frac{\alpha^2\beta}{(2\alpha -1)(1+\beta)} \right )^{\beta -1}\;.
$$
Since $\alpha>\alpha_1>1/2$ and $\beta>0$, by the last equality we conclude that $F''(t_1)<0$. Hence, the function $F(t)$ 
and, as a consequence, the function $f(t)$ attains its maximum on $[0,1]$ at the point $t_1\in(0, 1)$. 

Substituting $t_1$ from (\ref{EH_2BPAH}) in (\ref{EH_1BPAB}) and using (\ref{EH_2BPAB}), we arrive at (\ref{EH_4AA})
for the case $\alpha_{_1}<\alpha<\alpha_{_2}$.
\end{proof}

%%%%%%%%%%%%%%%%%%%%%%%%%%%%%%%%%%%%%%%%%%%%%%%%%%%%%%%%%%%%%%%%%
\section{The case $\alpha=0$}
%%%%%%%%%%%%%%%%%%%%%%%%%%%%%%%%%%%%%%%%%%%%%%%%%%%%%%%%%%%%%%%%%

In this section we consider integral (\ref{EH_1}) with $\alpha=0$ that is
$$
u(x)=k \int _{{\mathbb R}^{n-1}}\frac{f(y')}{|y-x|^\beta}\;dy'\;,
$$
where $x\in{\mathbb R}^{n} _+$, $\beta$ satisfies inequality (\ref{EQU_1}) and $f \in L^p({\mathbb R}^{n-1})$. 
Here we solve extremal problem (\ref{EH_4J}) with $\alpha=0$ and obtain the explicit value for $C_{0, \beta, p}$. 
Namely, we prove 

\begin{theorem} \label{P_1_0} Let $\alpha=0$ in $(\ref{EH_1})$ and let any of the following conditions holds:

{\rm (i)} $\beta\geq n-1$ and $p\in [1, \infty )$,

{\rm (ii)} $\beta > n-1$ and $p=\infty $,

{\rm (iii)} $\beta< n-1$ and $p\in [1, (n-1)/(n-1-\beta ))$.

Then for any $x\in{\mathbb R}^{n} _+$ the sharp constant $C_{0, \beta, p}$ in the inequality
\begin{equation} \label{EH_2A20}
\left |\nabla u(x)\right |\leq \frac{ C_{0, \beta, p}}{x_{n}^{2-n+\beta+((n-1)/p)}}\;\big|\!\big |f \big |\!\big |_p
\end{equation}
is given by $C_{0, \beta, 1}=|k|\beta$, and
\begin{equation} \label{EQ_expl}
C_{0, \beta, p}=|k|\beta\left \{\frac{\pi^{\frac{n-1}{2}}\Gamma\left ( \frac{(\beta -n +3)p+n-1}{2(p-1)} \right )}
{\Gamma\left ( \frac{(\beta +2)p}{2(p-1)} \right )} \right \}^{\frac{p}{p-1}}
\end{equation}
for $p>1$.

The constant $C_{0, \beta, p}$ is sharp under conditions of the Theorem also in the weaker inequality 
obtained from $(\ref{EH_2A20})$ by replacing $\nabla u$ by $\partial u/\partial x_n$.
\end{theorem}
\begin{proof} Let $\alpha=0$ in $(\ref{EH_1})$ and $p=1$. By (\ref{EH_4J}) and (\ref{EH_3HW3}),
\begin{equation} \label{EQ_P_1}
C_{0, \beta, 1}=|k|\beta\sup _{|\bs z|=1}\sup _{\sigma \in {\mathbb S}^{n-1}}
\big |\big (\bs e_{\sigma},\;\bs z \big )\big |\big |\big (\bs e_{\sigma}, \bs e_{n} \big )\big |^{\beta +1}\leq |k|\beta\;.
\end{equation}
On the other hand,
$$
C_{0, \beta, 1}\geq|k|\beta\sup _{\sigma \in {\mathbb S}^{n-1}}
\big |\big (\bs e_{\sigma},\;\bs e_{n} \big )\big |\big |\big (\bs e_{\sigma}, \bs e_{n} \big )\big |^{\beta +1}= |k|\beta\;,
$$
which, together with (\ref{EQ_P_1}), implies $C_1=|k|\beta$. We note that by (\ref{EH_2B}), 
$$
C_{0, \beta, 1}(\bs e_n)={|k|\beta}\sup _{\sigma \in {\mathbb S}^{n-1}_+}
\big (\bs e_{\sigma}, \bs e_{n} \big )^{\beta +2} =|k|\beta\;,
$$
that is $C_{0, \beta, 1}=C_{0, \beta, 1}(\bs e_n)$.

Let now $\alpha=0$ in $(\ref{EH_1})$ and $p> 1$. By (\ref{EH_4J}) and (\ref{EH_3HW1}) we have
\begin{equation} \label{EH_3HW1R}
C_{0, \beta, p}=\frac{|k|\beta}{2^{(p-1)/p}}\;\sup _{|\bs z|=1}
\left \{ \int _ {{\mathbb S}^{n-1}}\!\!
\big |\!\big (\bs e_{\sigma}, \bs z \big )\!\big |^{\frac{p}{p\!-\!1}}
\big |\big (\bs e_{\sigma}, \bs e_{n} \big )\big |^{\frac{(\beta - n+1)p+n}{p-1}}d\sigma \!\right \}^{\frac{p\!-\!1}{p}}\;.
\end{equation}

Let us denote by $\mu=p/(p-1)$ and $\lambda=((\beta - n+1)p+n)/(p-1)$ the powers in (\ref{EH_3HW1R}). 
Obviously, $\lambda >0$ for $\beta\geq n-1$ and any $p>1$. That is, $\lambda >0$ if condition (i) is satisfied.

If $\beta>n-1$ and $p=\infty$, then $\lambda=\beta-n+1>0$. Therefore, $\lambda >0$ if condition (ii) holds.

If $\beta <n-1$, then $\lambda >0$ for $p<n/(n-\beta -1)$. For $\beta -n+1<0$, by inequality (\ref{EQU_1}), 
we have $p<(n-1)/(n-\beta -1)$. So, $\lambda >0$ if condition (iii) is satisfied.

By H\"older's inequality, we obtain
\begin{equation} \label{Eq_H}
\int _ {{\mathbb S}^{n-1}}\!\!\big |\!\big (\bs e_{\sigma}, \bs z \big )\!\big |^\mu
\big |\big (\bs e_{\sigma}, \bs e_{n} \big )\big |^\lambda d\sigma\leq 
\left \{\int _ {{\mathbb S}^{n-1}}\!\!\big |\!\big (\bs e_{\sigma}, \bs z \big )\!\big |^{\mu\frac{\lambda+\mu}{\mu} }d\sigma\right \}
^{\frac{\mu}{\lambda+\mu}}
\left \{\int _ {{\mathbb S}^{n-1}}\!\!\big |\!\big (\bs e_{\sigma}, \bs e_n \big )\!\big |^{\lambda\frac{\lambda+\mu}{\lambda} }d\sigma\right \}
^{\frac{\lambda}{\lambda+\mu}}\;.
\end{equation}
Obviously, the value of the first integral on the right-hand side of the last inequality is independent of $\bs z$. Therefore, 
$$
\int _ {{\mathbb S}^{n-1}}\!\!\big |\!\big (\bs e_{\sigma}, \bs z \big )\!\big |^{\lambda+\mu }d\sigma=
\int _ {{\mathbb S}^{n-1}}\!\!\big |\!\big (\bs e_{\sigma}, \bs e_n \big )\!\big |^{\lambda+\mu}d\sigma\;,
$$
which, in view of (\ref{Eq_H}), implies
\begin{equation} \label{Eq_H1}
\int _ {{\mathbb S}^{n-1}}\!\!\big |\!\big (\bs e_{\sigma}, \bs z \big )\!\big |^\mu
\big |\big (\bs e_{\sigma}, \bs e_{n} \big )\big |^\lambda d\sigma\leq 
\int _ {{\mathbb S}^{n-1}}\!\!\big |\!\big (\bs e_{\sigma}, \bs e_n \big )\!\big |^{\lambda+\mu} d\sigma\;.
\end{equation}
On the other hand,
$$
\sup _{\bs z|=1} \int _ {{\mathbb S}^{n-1}}\!\!
\big |\!\big (\bs e_{\sigma}, \bs z \big )\!\big |^\mu
\big |\big (\bs e_{\sigma}, \bs e_{n} \big )\big |^\lambda d\sigma \geq 
\int _ {{\mathbb S}^{n-1}}\!\!\big |\!\big (\bs e_{\sigma}, \bs e_n \big )\!\big |^{\lambda+\mu} d\sigma\;,
$$
which together with (\ref{Eq_H1}) leads to
$$
\sup _{\bs z|=1} \int _ {{\mathbb S}^{n-1}}\!\!
\big |\!\big (\bs e_{\sigma}, \bs z \big )\!\big |^\mu
\big |\big (\bs e_{\sigma}, \bs e_{n} \big )\big |^\lambda d\sigma = 
\int _ {{\mathbb S}^{n-1}}\!\!\big |\!\big (\bs e_{\sigma}, \bs e_n \big )\!\big |^{\lambda+\mu} d\sigma\;.
$$
The last equality, in view of (\ref{EH_3HW1R}), implies
\begin{equation} \label{Eq_H2}
C_{0, \beta, p}=\frac{|k|\beta}{2^{(p-1)/p}}\left \{ \int _ {{\mathbb S}^{n-1}}\!\!
\big |\big (\bs e_{\sigma}, \bs e_{n} \big )\big |^{\frac{(\beta - n+2)p+n}{p-1}}d\sigma \!\right \}^{\frac{p\!-\!1}{p}}\;.
\end{equation}
Comparing (\ref{EH_3HW1}) with $\alpha=0$, $\bs z=\bs e_n$ and (\ref{Eq_H2}), we conclude that $C_{0, \beta, p}=C_{0, \beta, p}(\bs e_n)$.
This proves that the constant $C_{0, \beta, p}$ is sharp also in the weaker inequality 
obtained from $(\ref{EH_2A20})$ by replacing $\nabla u$ by $\partial u/\partial x_n$.

Evaluating the integral in (\ref{Eq_H2}), we find
\begin{eqnarray*}
& &\int _ {{\mathbb S}^{n-1}}\!\!\big |\big (\bs e_{\sigma}, \bs e_{n} \big )\big |^{\frac{(\beta - n+2)p+n}{p-1}}d\sigma=
2\omega_{n-1}\int_0^{\pi/2}\cos^{\frac{(\beta - n+2)p+n}{p-1}}\vartheta \sin^{n-2}\vartheta d\vartheta\\
& &\\
& &=\omega_{n-1}B\left (\frac{(\beta - n+3)p+n-1}{2(p-1)}, \frac{n-1}{2} \right )=
\frac{2\pi^{(n-1)/2}\Gamma\left (\frac{(\beta - n+3)p+n-1}{2(p-1)} \right )}{\Gamma\left (\frac{(\beta +2)p}{2(p-1)} \right )}\;,
\end{eqnarray*}
which together with (\ref{EH_2A2}), (\ref{EH_3J}), where $\alpha=0$, and (\ref{Eq_H2}) proves (\ref{EH_2A20}) and (\ref{EQ_expl}). 
\end{proof}

%%%%%%%%%%%%%%%%%%%%%%%%%%%%%%%%%%%%%%%%%%%%%%%%%%%%%%%%%%%%%%%%%
\section{Reduction of the extremal problem to finding of the 
supremum by parameter of a double integral}
%%%%%%%%%%%%%%%%%%%%%%%%%%%%%%%%%%%%%%%%%%%%%%%%%%%%%%%%%%%%%%%%%

The next assertion is based on the representation for ${\mathcal C}_{\alpha, \beta, p} (x)$,
obtained in Proposition \ref{P_1}.

\setcounter{theorem}{1}
\begin{proposition} \label{P_2} Let $f \in L^p({\mathbb R}^{n-1})$, $p>1$,
and let  $x $ be  an arbitrary point in ${\mathbb R}^{n} _+$. 
The sharp coefficient ${\mathcal C}_{\alpha, \beta, p} (x)$ in the inequality
\begin{equation} \label{EH_2_3}
|\nabla u(x)|\leq {\mathcal C}_{\alpha, \beta, p}( x)\big|\!\big |f \big |\!\big |_p
\end{equation}
is given by
\begin{equation} \label{EH_2A_3A}
{\mathcal C}_{\alpha, \beta, p}( x)= \frac{C_{\alpha, \beta, p} }{x_{n}^{2-n+\beta(1-\alpha)+((n-1)/p)}}\;,
\end{equation}
where
\begin{equation} \label{EH_1ABC}
C_{\alpha, \beta, p}\!=\!|k|\beta(\omega_{n-2})^{(p-1)/p}
\sup _{\gamma \geq 0}\;\frac{1}{\sqrt{1+\gamma^2}}\left \{ \int _ {0}^{\pi}d\varphi
\int _ {0}^{\pi/2}{\mathcal F}_{n,p}(\varphi, \vartheta ; \alpha, \beta, \gamma) \;
 d\vartheta\right \}^\frac{p-1}{p}.
\end{equation}
Here
\begin{equation} \label{EH_1AC}
{\mathcal F}_{n,p}(\varphi, \vartheta ; \alpha, \beta, \gamma)
=\big |{\mathcal G}(\varphi, \vartheta ; \alpha, \gamma) \big |^{p / (p-1)}
\cos^{((\beta-n)p+n)/(p-1)}\vartheta \sin^{n-2}\vartheta \sin^{n-3}\varphi
\end{equation}
with
\begin{equation} \label{E_AUX}
{\mathcal G}(\varphi, \vartheta ; \alpha, \gamma)= \cos^2 \vartheta -\alpha+\gamma\cos \vartheta
\sin \vartheta\cos \varphi \;.
\end{equation}
\end{proposition}
\begin{proof} 
The equality (\ref{EH_2A_3A}) for the sharp coefficient ${\mathcal C}_{\alpha, \beta, p}( x)$ in (\ref{EH_2_3}) 
was proved in Proposition \ref{P_1}. Since the integrand in (\ref{EH_3DA}) does not
change when $\bs z \in {\mathbb S}^{n-1}$ is replaced by $-\bs z$,
we may assume that $z_n=(\bs e_n, \bs z) > 0$ in (\ref{EH_4J}).

Let $\bs z'=\bs z-z_n\bs e_n$. Then $(\bs z', \bs e_n)=0$ and hence
$z^2_n+|\bs z'|^2=1$.
Analogously, with
$\sigma=(\sigma_1,\dots,\sigma_{n-1},\sigma_n) \in {\mathbb S}^{n-1}_+$,
we associate the vector $\bs \sigma '=\bs e_\sigma -\sigma_n\bs e_n$.

Using the equalities $(\bs \sigma ', \bs e_n)=0$,
$\sigma_n =\sqrt{1-|\bs \sigma '|^2}$
and $(\bs z', \bs e_n)=0$, we find an expression for
$(\alpha\bs e_n -(\bs e_{\sigma}, \bs e_n)\bs e_{\sigma},\; \bs z \big )$
as a function of $\bs \sigma'$:
\begin{eqnarray}
& &(\alpha\bs e_n - (\bs e_{\sigma}, \bs e_n)\bs e_{\sigma},\; \bs z \big )=
\alpha z_n-\sigma_n\big ( \bs e_{\sigma}, \bs z \big )
=\alpha z_n-\sigma_n
\big ( \bs \sigma '+\sigma_n \bs e_n,\; \bs z'+z_n\bs e_n \big )\nonumber\\
& &=\alpha z_n-\sigma_n
\big [ \big ( \bs \sigma ', \bs z' \big )+z_n \sigma_n\big ]
=-\big [(1-|\bs\sigma '|^2) -\alpha \big ]z_n-\sqrt{1-|\bs\sigma '|^2}\;
\big ( \bs \sigma ', \bs z' \big ) .
\label{EH_B}
\end{eqnarray}

Let ${\mathbb B}^{n}=\{ x=(x_1,\dots,x_n)\in {\mathbb R}^n: |x|< 1 \}$.
By (\ref{EH_3DA}) and (\ref{EH_B}), taking into account that
$d\sigma=d\sigma '/\sqrt{1-|\bs\sigma '|^2}$, we may write (\ref{EH_4J}) as
\begin{eqnarray} \label{EH_B1C}
\hspace{-15mm}& &C_{\alpha, \beta, p}=|k|\beta\sup _{\bs z \in {\mathbb S}^{n-1}_+}
\left \{ \int _ {{\mathbb B}^{n-1}}
\frac{{\mathcal H}_{\alpha, p}\big  ( |\bs \sigma '|,  (\bs \sigma ',  \bs z') \big  )
\big ( 1-|\bs\sigma '|^2\big )^{((\beta -n )p+n )/(2p-2)}}
{\sqrt{1-|\bs\sigma '|^2}}\;d\sigma ' \right \}^{\frac{p-1}{p}} \nonumber \\
\hspace{-15mm}& &\nonumber \\
\hspace{-15mm}& &=|k|\beta\sup _{\bs z \in {\mathbb S}^{n-1}_+}\!\!
\left \{\! \int _ {{\mathbb B}^{n-1}}
\!\!{\mathcal H}_{\alpha, p}\big  ( |\bs \sigma '|,  (\bs \sigma ',  \bs z') \big  )
\big ( 1\!-\!|\bs\sigma '|^2\big )^{((\beta -n -1)p+n+1 )/(2p-2)}d\sigma '\! \right \}^{\frac{p-1}{p}},
\end{eqnarray}
where
\begin{equation} \label{E_FH}
{\mathcal H}_{\alpha, p}\big  ( |\bs \sigma '|,  (\bs \sigma ',  \bs z') \big  )
=\Big | \big [(1-|\bs\sigma '|^2)-\alpha \big ]z_n+\sqrt{1-|\bs\sigma '|^2}\;
\big ( \bs \sigma ', \bs z' \big )\Big |^{p/(p-1)} .
\end{equation}

Using the well known formula (see e.g. \cite{PBM}, 3.3.2(3)),
$$
\int_{{\mathbb B}^{n}}g\big (|\bs x|, (\bs a, \bs x)\big )dx=\omega_{n-1}\int_0^1 r^{n-1} dr
\int_0^\pi g\big ( r, |\bs a|r \cos \varphi \big )\sin ^{n-2}\varphi \;d\varphi \;,
$$
we obtain
\begin{eqnarray*} 
& &\hspace{-17mm}\int _ {{\mathbb B}^{n-1}}
{\mathcal H}_{\alpha, p}\big  ( |\bs \sigma '|,  (\bs \sigma ',  \bs z') \big  )
\big ( 1-|\bs\sigma '|^2\big )^{((\beta -n -1)p+n+1 )/(2p-2)}\;d\sigma ' \nonumber\\
& &\hspace{-17mm}\\
& &\hspace{-17mm}=\omega_{n-2}\!\int^{1}_{0}\! r^{n-2}\big ( 1\!-\!r^2\big )^{((\beta -n -1)p+n+1 )/(2p-2)} dr
\!\int^{\pi}_{0}\!\!{\mathcal H}_{\alpha, p}
\big  (r, r|\bs z'|\cos  \varphi \big  )\sin ^{n-3}\varphi d\varphi\;.
\end{eqnarray*}
Making the change of variable $r=\sin \vartheta $ on the right-hand side of the last equality, we find
\begin{eqnarray} \label{E_PBMD}
& &\hspace{-7mm}\int _ {{\mathbb B}^{n-1}}
{\mathcal H}_{\alpha, p}\big  ( |\bs \sigma '|,  (\bs \sigma ',  \bs z') \big  )
\big ( 1-|\bs\sigma '|^2\big )^{((\beta -n -1)p+n+1 )/(2p-2)}\;d\sigma ' \\
& &\hspace{-7mm}\nonumber\\
& &\hspace{-7mm}=\omega _{n-2}\int^{\pi}_{0}\!\! \sin ^{n-3}\varphi d\varphi \!\int^{\pi/2}_{0}\!\!
{\mathcal H}_{\alpha, p}\big  (\sin \vartheta,\; |\bs z'|\sin \vartheta\cos  \varphi \big  )
\sin ^{n-2}\vartheta \cos^{\frac{(\beta -n )p+n }{p-1}}\vartheta d\vartheta
\;,\nonumber
\end{eqnarray}
where, by (\ref{E_FH}),
$$
{\mathcal H}_{\alpha, p}\big  (\sin \vartheta,\; |\bs z'|\sin \vartheta\cos  \varphi \big  )=
\Big |\big ( \cos^2 \vartheta -\alpha \big )z_n\!+\!|\bs z'|\cos \vartheta
\sin \vartheta\cos \varphi \Big |^{p/(p-1)}.
$$
Introducing here the parameter $\gamma =|\bs z'|/z_n$ and using
the equality $|\bs z'|^2+z^2_n=1$, we obtain
\begin{equation} \label{E_FHGH}
{\mathcal H}_{\alpha, p}\big  (\sin \vartheta,\; |\bs z'|\sin \vartheta\cos  \varphi \big  )=
(1+\gamma^2)^{-p/(2p-2)}\big |{\mathcal G}(\varphi, \vartheta ; \alpha, \gamma)\big |^{p/(p-1)},
\end{equation}
where ${\mathcal G}(\varphi, \vartheta ; \alpha, \gamma)$ is given by (\ref{E_AUX}).

By (\ref{EH_B1C}), taking into account (\ref{E_PBMD}) and (\ref{E_FHGH}), we arrive at (\ref{EH_1ABC}).
\end{proof}

%%%%%%%%%%%%%%%%%%%%%%%%%%%%%%%%%%%%%%%%%%%%%%%%%%%%%%%%%%%%%%%%%
\section{The case $p=2$}
%%%%%%%%%%%%%%%%%%%%%%%%%%%%%%%%%%%%%%%%%%%%%%%%%%%%%%%%%%%%%%%%%

In the next assertion we obtain the explicit formula for $C_{\alpha, \beta, 2}$.

\setcounter{theorem}{2}
\begin{theorem} \label{C_2_2} Let $f \in L^2({\mathbb R}^{n-1})$,
and let  $x $ be  an arbitrary point in ${\mathbb R}^{n} _+$. 
The sharp coefficient ${\mathcal C}_{\alpha, \beta, 2} (x)$ in the inequality
\begin{equation} \label{EH_2_2.1}
|\nabla u(x)|\leq {\mathcal C}_{\alpha, \beta, 2}( x)\big|\!\big |f \big |\!\big |_2
\end{equation}
is given by
\begin{equation} \label{EH_2_2.2}
{\mathcal C}_{\alpha, \beta, 2}( x)= \frac{C_{\alpha, \beta, 2}}{x_{n}^{\beta(1-\alpha)+((3-n)/2)}}\;,
\end{equation}
where
\begin{equation} \label{EH_2_2.3}
C_{\alpha, \beta, 2}=|k|\beta\left \{ \frac{\pi^{(n-1)/2}
\Gamma\left (\frac{2\beta+3-n}{2} \right )}{\Gamma(\beta+2)}
\left [  \frac{2\alpha^2 \beta(\beta+1)}{2\beta+1-n} -2\alpha(\beta+1)+ \frac{2\beta+3-n}{2}\right ]\right \}^{1/2}
\end{equation}
for $(n-1)/2<\beta\leq n-1$. The same formula for $C_{\alpha, \beta, 2}$ holds for $\beta>n-1$ and
$$
\alpha\leq \alpha_{_1}=\frac{(1+\beta)(2\beta+1-n)-\sqrt{(1+\beta)(2\beta+1-n)(\beta+1-n)}}{2\beta(1+\beta)}\;,
$$
or
$$
\alpha\geq \alpha_{_2}=\frac{(1+\beta)(2\beta+1-n)+\sqrt{(1+\beta)(2\beta+1-n)(\beta+1-n)}}{2\beta(1+\beta)}\;.
$$

If $\beta>n-1$ and $\alpha_1<\alpha< \alpha_2$, then
\begin{equation} \label{EH_2_2.6}
C_{\alpha, \beta, 2}=|k|\beta\left \{ \frac{\pi^{(n-1)/2}
\Gamma\left (\frac{2\beta+3-n}{2} \right )}{2\Gamma(\beta+2)}\right \}^{1/2}\;.
\end{equation}

If ${\rm (i)}$ $(n-1)/2<\beta\leq n-1$ or ${\rm (ii)}$ $\beta>n-1$, $\alpha \leq\alpha_1$ or  $\alpha\geq\alpha_2$, 
then the coefficient ${\mathcal C}_{\alpha, \beta, 2}( x)$ is sharp under conditions of the Theorem also in
the weaker inequality obtained from $(\ref{EH_2_2.1})$ by replacing $\nabla u$ by
$\partial u/\partial x_n$.
\end{theorem}
\begin{proof} The equality (\ref{EH_2_2.2}) for the sharp coefficient ${\mathcal C}_{\alpha, \beta, 2}(x)$ in (\ref{EH_2_2.1}) 
was proved in Proposition \ref{P_1}. By (\ref{EH_1ABC}), (\ref{EH_1AC}) and (\ref{E_AUX}),
\begin{equation} \label{EH_9ABC}
C_{\alpha, \beta, 2}= |k|\beta\sqrt{\omega_{n-2}}\;\sup _{\gamma \geq 0}\;\frac{1}
{\sqrt{1+\gamma^2}}\left \{ \int _ {0}^{\pi}d\varphi\int _ {0}^{\pi/2}\!\!
{\mathcal F}_{n,2}(\varphi, \vartheta ; \alpha, \beta, \gamma)\; d\vartheta\right \}^{1/2},
\end{equation}
where
$$
\!{\mathcal F}_{n,2}(\varphi, \vartheta ; \alpha, \beta, \gamma)\!=\!
\big (\cos^2 \vartheta \!-\!\alpha \!+\!\gamma\cos \vartheta
\sin \vartheta\cos \varphi \big )^{2}\! \cos^{2\beta-n}\!\vartheta \!\sin^{n-2}\!\vartheta \sin ^{n-3}\!\varphi.
$$
The last equality and (\ref{EH_9ABC}) imply
\begin{equation} \label{EH_11ABCD}
C_{\alpha, \beta, 2}=|k|\beta\sqrt{\omega_{n-2}}\;\sup _{\gamma \geq 0}\;\frac{1}{\sqrt{1+\gamma^2}}
\left \{{\mathcal I}_1+ \gamma^2 {\mathcal I}_2 \right \}^{1/2},
\end{equation}
where
\begin{eqnarray} \label{EH_9B}
\hspace{-15mm}& &{\mathcal I}_1=\int _ {0}^{\pi} \sin ^{n-3}\varphi\;d\varphi
\int _ {0}^{\pi/2}\big (\cos^2 \vartheta -\alpha\big )^2\sin^{n-2}
\vartheta \cos^{2\beta -n} \vartheta \;d\vartheta \nonumber\\
\hspace{-15mm}& &\nonumber\\
\hspace{-15mm}& &=\frac{ \sqrt{\pi}\;\Gamma\left (\frac{n-2}{2} \right )\Gamma\left (\frac {2\beta+3-n}{2}
\right )}{2\Gamma(\beta+2)}\left \{ \frac{2\alpha^2 \beta(\beta+1)}{2\beta+1-n} -2\alpha(\beta+1)+ \frac{2\beta+3-n}{2} \right \} 
\end{eqnarray}
and
\begin{eqnarray} \label{EH_9BA}
\hspace{-8mm}{\mathcal I}_2&=&\int _ {0}^{\pi}\sin ^{n-3}\varphi \cos^2 \varphi\; d\varphi
\int _ {0}^{\pi/2}\sin^n \vartheta \cos^{2(\beta+1)-n} \vartheta
\;d\vartheta \nonumber\\
\hspace{-8mm}& &\nonumber\\
\hspace{-8mm}&=&\frac{\sqrt{\pi}\;\Gamma\left (\frac{n-2}{2} \right )
\Gamma\left (\frac{2\beta+3-n}{2} \right )}{4\Gamma(\beta+2)}.
\end{eqnarray}
By (\ref{EH_11ABCD}) we have
\begin{equation} \label{EQ_11ABCD}
C_{\alpha, \beta, 2}=|k|\beta\sqrt{\omega_{n-2}}\;\max \big \{ {\mathcal I}_1^{1/2}, {\mathcal I}_2^{1/2} \big  \}.
\end{equation}

Further, by (\ref{EH_9B}) and (\ref{EH_9BA}),
\begin{eqnarray} \label{EQ_10ABE}
\frac{{\mathcal I}_1}{{\mathcal I}_2}-1&=&4\alpha^2\frac{ \beta(\beta+1)}{2\beta+1-n} -4\alpha(\beta+1)+ 2(\beta+1)-n\nonumber\\
& &\nonumber\\
&=&\frac{4\alpha^2\beta(\beta+1)-4\alpha(\beta+1)(2\beta+1-n)+(2\beta+2-n)(2\beta+1-n)}{2\beta+1-n}\;.
\end{eqnarray}
We note that, by (\ref{EQ_11ABCD}) with $p=2$, $2\beta-n+1 >0$. By
$$
f(\alpha)=4\alpha^2\beta(\beta+1)-4\alpha(\beta+1)(2\beta+1-n)+(2\beta+2-n)(2\beta+1-n)
$$
we denote the numerator of fraction (\ref{EQ_10ABE}).
The roots of the equation $f(\alpha)=0$ are
\begin{equation} \label{EQ_10ABF}
\alpha_{_{1,2}}=\frac{(\beta+1)(2\beta+1-n)\pm \sqrt{(\beta+1)(2\beta+1-n)(\beta+1-n)}}{2\beta(\beta+1)}\;.
\end{equation}

It follows from (\ref{EQ_10ABE}) and (\ref{EQ_10ABF}) that ${\mathcal I}_1\geq {\mathcal I}_2$ for $\beta+1-n\leq 0$.
Combining the last condition for $\beta$ with inequality $\beta> (n-1)/2$ and taking into account (\ref{EH_9B}), 
(\ref{EQ_11ABCD}), we arrive at formula (\ref{EH_2_2.3}) for the case  $(n-1)/2<\beta\leq n-1$. 

Now, let $\beta >n-1$. Then, by (\ref{EQ_10ABE}),
$$
\frac{{\mathcal I}_1}{{\mathcal I}_2}-1\geq 0
$$
for $\alpha \leq \alpha_{_{1}}$ or $\alpha \geq \alpha_{_{2}}$, and
$$
\frac{{\mathcal I}_1}{{\mathcal I}_2}-1< 0
$$
for $\alpha_{_{1}}<\alpha <\alpha_{_{2}}$. This, by (\ref{EQ_11ABCD}), proves (\ref{EH_2_2.3}) for 
$\alpha \leq \alpha_{_{1}}$ or $\alpha \geq \alpha_{_{2}}$ and (\ref{EH_2_2.6}) for $\alpha_{_{1}}<\alpha <\alpha_{_{2}}$.

\smallskip
In conclusion, we note that supremum in (\ref{EH_11ABCD}) is attained for $\gamma=0$ in two cases: (i) $(n-1)/2<\beta\leq n-1$,
(ii) $\beta >n-1$ and $\alpha \leq \alpha_{_{1}}$ or $\alpha \geq \alpha_{_{2}}$. Taking into account that $\gamma =|\bs z'|/z_n$,
we conclude that $C_{\alpha, \beta, 2}=C_{\alpha, \beta, 2}(\bs e_n)$ for these cases. 
This proves that the coefficient ${\mathcal C}_{\alpha, \beta, 2}(x)$ is sharp under conditions of the Theorem also in the weaker inequality 
obtained from $(\ref{EH_2_2.1})$ by replacing $\nabla u$ by $\partial u/\partial x_n$.
\end{proof}

%%%%%%%%%%%%%%%%%%%%%%%%%%%%%%%%%%%%%%%%%%%%%%%%%%%%%%%%%%%%%%%%%%%%%%%
\section{The case $p=\infty$}
%%%%%%%%%%%%%%%%%%%%%%%%%%%%%%%%%%%%%%%%%%%%%%%%%%%%%%%%%%%%%%%%%%%%%%%

This section is devoted to the case $p=\infty$ with some restrictions on $\alpha$ and $\beta$. 
In the assertion below we obtain the explicit formula for $C_{\alpha, \beta, \infty }$ with 
any fixed $\beta>n-1$ and sufficiently large $\alpha >1$.
We note that inequality $\beta>n-1$ follows from (\ref{EQU_1}) with $p=\infty$.

\begin{theorem} \label{SP_2AAA} Let $f\in L^\infty({\mathbb R}^{n-1})$,
and let  $x $ be  an arbitrary point in ${\mathbb R}^n _+$. Let $\beta$ be a fixed and let 
$\alpha_n(\beta )$ be the root from the interval $(1, +\infty )$ of the equation  
\begin{equation} \label{EH_XBB}
\frac{2\Gamma\left ( \frac{\beta -n}{2}+1\right )}{\sqrt{\pi}\big ( \beta (\alpha-1)+n-1 \big )
\Gamma\left ( \frac{ \beta -n+1}{2}\right ) }
=\frac{\alpha-1}{1+\sqrt{1+(\alpha-1)^2}}
\end{equation} 
with respect to $\alpha$. 

If $\alpha \geq \alpha_n(\beta )$, then the sharp coefficient ${\mathcal C}_{\alpha, \beta, \infty}(x)$ in the inequality
\begin{equation} \label{EH_XAA}
|\nabla u(x)|\leq {\mathcal C}_{\alpha, \beta, \infty}( x)\big|\!\big |f \big |\!\big |_\infty
\end{equation}
is given by
\begin{equation} \label{EH_YAA}
{\mathcal C}_{\alpha, \beta, \infty }( x)=\frac{C_{\alpha, \beta, \infty}}{x_n^{2-n+\beta(1-\alpha)}}\;,
\end{equation}
where
\begin{equation} \label{EH_Y001AA}
C_{\alpha, \beta,\infty}=|k|\frac{\pi^{(n-1)/2}\;\Gamma\left ( \frac{\beta -n+1}{2}\right )}{
\Gamma\left ( \frac{\beta}{2}\right )}\big ( (\alpha-1)\beta+n-1  \big )\;.
\end{equation}

Under conditions of the Theorem, absolute value of the derivative of $u$
with respect to the normal to the boundary of the half-space at any
$x \in {\mathbb R}^n_+$ has the same supremum as $|\nabla u(x)|$.
\end{theorem}

\begin{proof} First of all, we show that equation (\ref{EH_XBB}) has only one $\alpha$-root $\alpha_n(\beta )$ on the 
interval $(1, +\infty )$ for any fixed $\beta >n-1$. In fact,
the function
\begin{equation} \label{Eq_f}
f(\alpha)=\frac{2\Gamma\left ( \frac{\beta -n}{2}+1\right )}{\sqrt{\pi}\big ( \beta (\alpha-1)+n-1 \big )
\Gamma\left ( \frac{ \beta -n+1}{2}\right ) }
\end{equation}
decreases, and the function
\begin{equation} \label{Eq_g}
g(\alpha)=\frac{\alpha-1}{1+\sqrt{1+(\alpha-1)^2}}
\end{equation}
increases on the interval $[1, \infty )$. The functions $f$, $g$ are continuous, $f(1) >0, g(1)=0$, and
$$
\lim_{\alpha\rightarrow +\infty } f(\alpha)=0,\;\;\;\;\lim_{\alpha\rightarrow +\infty } g(\alpha)=1\;.
$$
So, the existence and uniqueness of the $\alpha$-root $\alpha_n(\beta )$ of equation (\ref{EH_XBB}) 
on the interval $(1, +\infty )$ are proven.

The equality (\ref{EH_YAA}) for the sharp coefficient ${\mathcal C}_{\alpha, \beta, \infty}(x)$ in (\ref{EH_XAA}) 
was proved in Proposition \ref{P_1}.
We pass to the limit as $p \rightarrow \infty $ in (\ref{EH_1ABC}) and (\ref{EH_1AC}). This results at
\begin{equation} \label{EH_inf1AA}
C_{\alpha, \beta, \infty}=|k|\beta\sup _{\gamma \geq 0}\;
\frac{\omega _{n-2}}{\sqrt{1+\gamma^2}}\!\!\int _ {0}^{\pi}\sin ^{n-3}
\varphi d\varphi \!\!\int _ {0}^{\pi/2}
\big |{\mathcal G}(\varphi, \vartheta ; \alpha, \gamma)\big |\cos^{\beta -n}\vartheta \sin^{n-2}\vartheta d\vartheta ,
\end{equation}
where ${\mathcal G}(\varphi, \vartheta ;\alpha, \gamma)$ is defined by (\ref{E_AUX}).

Suppose that $\beta>n-1$ and $\alpha\geq \alpha_n(\beta )$ are fixed. We introduce three integrals
\begin{equation} \label{Eq_a1}
J(\gamma)=|k|\beta\omega _{n-2}\int _ {0}^{\pi}\sin ^{n-3}
\varphi d\varphi \!\!\int _ {0}^{\pi/2}
\big |\alpha-\cos^2 \vartheta -\gamma\cos \vartheta
\sin \vartheta\cos \varphi \big |\cos^{\beta -n}\vartheta \sin^{n-2}\vartheta d\vartheta\;,
\end{equation}
\begin{equation} \label{Eq_a3}
J_1=|k|\beta\omega _{n-2}\int _ {0}^{\pi}\sin ^{n-3}
\varphi d\varphi \!\!\int _ {0}^{\pi/2}
\big |\alpha-\cos^2 \vartheta  \big |\cos^{\beta -n}\vartheta \sin^{n-2}\vartheta d\vartheta\;, 
\end{equation}
and
$$
J_2=|k|\beta\omega _{n-2}\int _ {0}^{\pi}\sin ^{n-3}
\varphi |\cos \varphi | d\varphi \!\!\int _ {0}^{\pi/2}
\cos^{\beta -n+1}\vartheta \sin^{n-1}\vartheta d\vartheta .
$$
We note that 
\begin{equation} \label{Eq_a4A}
\frac{J(\gamma)}{\sqrt{1+\gamma^2}}\leq\frac{J_1+\gamma J_2}{\sqrt{1+\gamma^2}}\;.
\end{equation}

Calculating $J_1$ and $J_2$, we obtain
\begin{equation} \label{Eq_a5}
J_1=|k|\frac{\pi^{(n-1)/2}\;\Gamma\left ( \frac{\beta -n+1}{2}\right )}{
\Gamma\left ( \frac{\beta}{2}\right )}\big ( (\alpha-1)\beta+n-1 \big )\;,
\end{equation}
and 
\begin{equation} \label{Eq_a6}
J_2=|k|\frac{2\pi^{(n-2)/2}\;\Gamma\left ( \frac{\beta -n}{2}+1\right )}
{\Gamma\left ( \frac{\beta}{2}\right )}\;.
\end{equation}

It follows from (\ref{Eq_a5}) and (\ref{Eq_a6}) that
\begin{equation} \label{Eq_a7}
\frac{J_2}{J_1}=\frac{2\Gamma\left ( \frac{\beta -n}{2}+1\right )}{\sqrt{\pi}\big ( (\alpha-1)\beta+n-1 \big )
\Gamma\left ( \frac{ \beta -n+1}{2}\right ) }\;.
\end{equation}
We note that the right-hand side of the last equality coincides with the function $f(\alpha)$, defined by (\ref{Eq_f}).
Since $f(\alpha)\leq g(\alpha)$ for $\alpha\geq \alpha_n(\beta )$ and $g(\alpha)<1$, by (\ref{Eq_a7}) we conclude that 
\begin{equation} \label{Eq_a7A}
\frac{J_2}{J_1}=f(\alpha)<1
\end{equation}
for $\alpha\geq \alpha_n(\beta )$. We find the interval of $\gamma$ for which the inequality 
\begin{equation} \label{Eq_a8}
\frac{J_1+\gamma J_2}{\sqrt{1+\gamma^2}}\leq J_1
\end{equation}
holds. Solving inequality (\ref{Eq_a8}) with respect to $\gamma$, we obtain
\begin{equation} \label{Eq_a9}
\gamma \geq \frac{2J_1J_2}{J_1^2 - J_2^2}=\frac{2(J_2/J_1)}{1-(J_2/J_1)^2}\;.
\end{equation}
We denote
$$
\gamma _0=\frac{2(J_2/J_1)}{1-(J_2/J_1)^2}\;.
$$
By (\ref{Eq_a4A}), (\ref{Eq_a8}) and (\ref{Eq_a9}),
\begin{equation} \label{Eq_a9AB}
\frac{J(\gamma)}{\sqrt{1+\gamma^2}}\leq J_1\;\;{\rm for}\;\;\gamma \geq \gamma_0\;.
\end{equation}

Now, we show that $\alpha -1-\gamma_0 \geq 0$ for $\alpha\geq \alpha_n(\beta )$. Taking into account that 
$f(\alpha)\leq g(\alpha)$ for $\alpha\geq \alpha_n(\beta )$, by (\ref{Eq_g}) and (\ref{Eq_a7A}) we arrive at inequality
\begin{equation} \label{Eq_a9ABC}
\frac{J_2}{J_1}\leq \frac{\alpha-1}{1+\sqrt{1+(\alpha-1)^2}}\;.
\end{equation}
Using (\ref{Eq_a9ABC}), after calculations we obtain
$$
\gamma_0=\frac{2(J_2/J_1)}{1-(J_2/J_1)^2}\leq \alpha -1\;,
$$
which proves the inequality $\alpha -1-\gamma_0 \geq 0$ for $\alpha\geq \alpha_n(\beta )$.

Let $0\leq\gamma\leq \gamma_0$. Taking into account that $\alpha -1-\gamma \geq 0$, by (\ref{Eq_a1}) and (\ref{Eq_a3}) 
we have $J(\gamma)=J_1$. Hence,
$$
\frac{J(\gamma)}{\sqrt{1+\gamma^2}}=\frac{J_1}{\sqrt{1+\gamma^2}}\leq J_1\;\;{\rm for}\;\;0\leq\gamma \leq \gamma_0\;,
$$
which together with (\ref{Eq_a9AB}) leads to inequality
$$
\frac{J(\gamma)}{\sqrt{1+\gamma^2}}\leq J_1
$$
for any $\gamma\geq 0$. Therefore, in view of (\ref{EH_inf1AA})-(\ref{Eq_a3}), we obtain
\begin{equation} \label{Eq_a9ABCD}%$$
C_{\alpha, \beta, \infty}=\sup_{\gamma \geq 0}\frac{J(\gamma)}{\sqrt{1+\gamma^2}}\leq J_1=J(0)\;,
\end{equation}%$$
which together with inequality
$$
\sup_{\gamma\geq 0}\frac{J(\gamma)}{\sqrt{1+\gamma^2}}\geq J(0)
$$
results at
$$
C_{\alpha, \beta, \infty}=J(0)=J_1\;.
$$
In view of (\ref{Eq_a5}), the last equality proves (\ref{EH_Y001AA}). 
Since $\gamma =|\bs z'|/z_n$ and the supremum with respect to $\gamma$ in (\ref{Eq_a9ABCD}) is attained at $\gamma=0$, we conclude that the coefficient ${\mathcal C}_{\alpha, \beta, \infty}(x)=C_{\alpha, \beta, \infty}x_n^{n-2+\beta(\alpha-1)}$ is sharp under conditions of the Theorem also in the weaker inequality, obtained from $(\ref{EH_XAA})$ by replacing $\nabla u$ by $\partial u/\partial x_n$.
\end{proof}

\begin{remark} As an example, we give a number of values of $\alpha_n(\beta)$, obtained by numerical solution of equation (\ref{EH_XBB}):
$$
\alpha_3(2.5)\approx 1.2865,\;\;\alpha_3(3)\approx 1.4101,\;\;\alpha_3(3.5)\approx 1.4788,\;\;\alpha_3(4)\approx 1.521,\;\;\alpha_3(4.5)\approx 1.5482,\;\;\alpha_3(5)\approx 1.5664,
$$
$$
\alpha_4(3.5)\approx 1.207,\;\;\alpha_4(4)\approx 1.3079,\;\;\alpha_4(4.5)\approx 1.3698,\;\;\alpha_4(5)\approx 1.4115,\;\;\alpha_4(5.5)\approx 1.4413,\;\;\alpha_4(6)\approx 1.4631,
$$
$$
\alpha_5(4.5)\approx 1.1623,\;\;\alpha_5(5)\approx 1.2469,\;\;\alpha_5(5.5)\approx 1.3016,\;\;\alpha_5(6)\approx 1.3403,\;\;\alpha_5(6.5)\approx 1.3693,\;\;\alpha_5(7)\approx 1.3917,
$$
$$
\alpha_6(5.5)\approx 1.1316,\;\;\alpha_6(6)\approx 1.2063,\;\;\alpha_6(6.5)\approx 1.2548,\;\;\alpha_6(7)\approx 1.2903,\;\;\alpha_6(7.5)\approx 1.3176,\;\;\alpha_6(8)\approx 1.3393.
$$
\end{remark}

\bigskip
The representation for $C_{\alpha, \beta, \infty}$ with $\alpha\in (0, 1]$, obtained in the following auxiliary assertion, will be used later.
 
\begin{lemma} \label{L_2A} Let $0<\alpha \leq 1$ and $\beta>n-1$. Let $f\in L^\infty({\mathbb R}^{n-1})$,
and let  $x $ be  an arbitrary point in ${\mathbb R}^n _+$.
The sharp coefficient ${\mathcal C}_{\alpha, \beta, \infty} (x)$ in the inequality
\begin{equation} \label{EH_XAAB}
|\nabla u(x)|\leq {\mathcal C}_{\alpha, \beta, \infty}( x)\big|\!\big |f \big |\!\big |_\infty
\end{equation}
is given by
\begin{equation} \label{EH_XAAC}
{\mathcal C}_{\alpha, \beta, \infty} (x)=C_{\alpha, \beta, \infty} \; x_n^{n-2+\beta(\alpha-1)},
\end{equation}
where
\begin{equation} \label{EA_idnX}
C_{\alpha, \beta, \infty}=|k|\beta\sup _{\gamma \geq 0}\frac{ \omega _{n-2}}{\sqrt{1+\gamma^2}}\left \{ -c_{n, \beta}(\alpha)+2\int _ {0}^{\pi} 
P\big ( h_\gamma(\varphi )\big )\sin ^{n-3}\varphi d\varphi\right \} .
\end{equation}
Here
\begin{equation} \label{EC_idnX}
P(z)=\frac{(2\alpha)^{\beta-n+1}z^{n-1}}{\beta(4\alpha^2 +z^2)^{\beta/2}}+\int _ {0}^{\arctan \frac{z}{2\alpha}} 
\left \{ \!\frac{\beta\!-\!n\!+\!1}{\beta}\!-\!\alpha \!+\!\gamma\cos \varphi\cos \vartheta\sin\vartheta
\!\right \}\!\cos ^{\beta\!-\!n}\vartheta\sin^{n\!-\!2}\vartheta d\vartheta\;,
\end{equation}
\begin{equation} \label{EC_iii}
c_{n, \beta}(\alpha)=\left (\frac{\beta-n+1}{\beta} -\alpha \right )\frac{\sqrt{\pi}\Gamma\left (\frac{n-2}{2} \right )
\Gamma\left (\frac{\beta-n+1}{2} \right )}{2\Gamma\left (\frac{\beta}{2} \right )}\;,
\end{equation}
and
\begin{equation} \label{EC_jjj}
h_\gamma(\varphi)=\gamma \cos \varphi + \Big (\gamma ^2\cos ^2\varphi +4\alpha(1-\alpha )\Big )^{1/2}\;.
\end{equation}
\end{lemma}
\begin{proof} The formula (\ref{EH_XAAC}) for the sharp coefficient ${\mathcal C}_{\alpha, \beta, \infty}(x)$ in (\ref{EH_XAAB}) 
was proved in Proposition \ref{P_1}. By (\ref{EH_inf1AA}) and (\ref{E_AUX}) we have
\begin{equation} \label{EH_inf1}
C_{\alpha, \beta, \infty}=|k|\beta\sup _{\gamma \geq 0}\;
\frac{\omega _{n-2}}{\sqrt{1+\gamma^2}}\!\!\int _ {0}^{\pi}\sin ^{n-3}
\varphi d\varphi \!\!\int _ {0}^{\pi/2}
\big |{\mathcal G}(\varphi, \vartheta ; \alpha, \gamma)\big |\cos^{\beta -n}\vartheta \sin^{n-2}\vartheta d\vartheta ,
\end{equation}
where 
\begin{equation} \label{EH_inf_01AABC}
{\mathcal G}(\varphi, \vartheta ;\alpha, \gamma)=\cos^2 \vartheta-\alpha+\gamma\cos \vartheta
\sin \vartheta\cos \varphi .
\end{equation}

First, we calculate the integral
\begin{eqnarray} \label{eqn_001}
c_{n, \beta}(\alpha)&\!\!\!=\!\!\!&\int _ {0}^{\pi}\sin ^{n-3}\varphi d\varphi\int _ {0}^{\pi/2}
{\mathcal G}(\varphi, \vartheta ;\alpha, \gamma)\cos ^{\beta-n} \vartheta\sin^{n-2}\vartheta d\vartheta\nonumber\\
&\!\!\!=\!\!\!&\int _ {0}^{\pi}\!\!\sin ^{n-3}\varphi d\varphi\!\!\int _ {0}^{\pi/2}\!\!
\big \{ \cos^2 \vartheta\! -\!\alpha\!+\!\gamma\cos \vartheta
\sin \vartheta\cos \varphi \big \}\cos ^{\beta-n} \vartheta\sin^{n-2}\vartheta d\vartheta\nonumber\\
&\!\!\!=\!\!\!&\int _ {0}^{\pi}\!\!\sin ^{n-3}\varphi d\varphi\!\!\int _ {0}^{\pi/2}\!\!
\big (\cos^2 \vartheta\! -\!\alpha\big )\cos ^{\beta-n} \vartheta\sin^{n-2}\vartheta d\vartheta\nonumber\\
&\!\!\!=\!\!\!&\left (\frac{\beta-n+1}{\beta} -\alpha \right )\frac{\sqrt{\pi}\Gamma\left (\frac{n-2}{2} \right )
\Gamma\left (\frac{\beta-n+1}{2} \right )}{2\Gamma\left (\frac{\beta}{2} \right )}\;.
\end{eqnarray}

Now, we are looking for a solution of the equation
\begin{equation} \label{EH_inf2}
\cos^2 \vartheta -\alpha+\gamma\cos \vartheta \sin \vartheta\cos \varphi =0
\end{equation}
as a  function $\vartheta$ of $\varphi$. We can rewrite  (\ref{EH_inf2}) as the
second order equation in $\tan \vartheta$:
$$
-\alpha \tan^2\vartheta +\gamma \cos \varphi\tan \vartheta  +(1-\alpha)=0.
$$
Since $0\leq \vartheta \leq \pi/2$, we find that the nonnegative root of this equation is
\begin{equation} \label{EH_inf3}
\vartheta _\gamma(\varphi)=\arctan \frac{h_\gamma(\varphi)}{2\alpha}\;,
\end{equation}
where
\begin{equation} \label{EQN_1}
h_\gamma(\varphi)=\gamma \cos \varphi + \Big (\gamma ^2\cos ^2\varphi +4\alpha(1-\alpha )\Big )^{1/2}\;.
\end{equation}

We calculate the integral
\begin{eqnarray} \label{eqn_002}
H_n(\varphi, \psi; \alpha, \gamma)\!\!\!\!\!\!&=&\!\!\!\!\!\!\int _ {0}^{\psi}
{\mathcal G}(\varphi, \vartheta ; \alpha, \gamma)\cos ^{\beta -n} \vartheta\sin^{n-2}\vartheta d\vartheta \nonumber\\
\!\!\!\!\!\!&=&\!\!\!\!\!\!\int _ {0}^{\psi}\big (\cos^2 \vartheta -\alpha+\gamma\cos \vartheta
\sin \vartheta\cos \varphi  \big )\cos ^{\beta -n} \vartheta\sin^{n-2}\vartheta d\vartheta \nonumber\\
\!\!\!\!\!\!&=&\!\!\!\!\!\!\int _ {0}^{\psi} \big (\cos^2 \vartheta -\alpha \big )\cos ^{\beta -n} \vartheta\sin^{n-2}\vartheta d\vartheta+
\gamma\cos \varphi\int_{0}^{\psi}\cos ^{\beta-n+1}\vartheta\sin^{n-1}\vartheta d\vartheta\nonumber\\
\!\!\!\!\!\!\!&=&\!\!\!\!\!\!\frac{\sin^{n\!-\!1}\psi \cos ^{\beta\!-\!n\!+\!1}\psi}{\beta}\!+\!
\int _ {0}^{\psi}\!\!\!\left \{ \!\!\left (\frac{\beta\!-\!n\!+\!1}{\beta}\!-\!\alpha \right )\!+\!\gamma\cos \varphi\cos \vartheta\sin\vartheta
\!\right \}\!\!\cos ^{\beta\!-\!n}\vartheta\sin^{n\!-\!2}\vartheta d\vartheta .
\end{eqnarray}

Obviously, ${\mathcal G}(\varphi, \vartheta ; \alpha, \gamma) \geq 0$ for $0\leq\vartheta\leq\vartheta _\gamma(\varphi)$ and
${\mathcal G}(\varphi, \vartheta ; \alpha, \gamma) <0$ for $\vartheta _\gamma(\varphi)<\vartheta \leq\pi/2$. Hence,
\begin{eqnarray} \label{EH_inf4_01}
& &\int _ {0}^{\pi}\sin ^{n-3}\varphi d\varphi \!\!\int _ {0}^{\pi/2}
\big |{\mathcal G}(\varphi, \vartheta ; \alpha, \gamma)\big | \cos ^{\beta-n} \vartheta\sin^{n-2}\vartheta d\vartheta \nonumber\\
& &\nonumber\\
& &=\int _ {0}^{\pi}\sin ^{n-3}\varphi d\varphi \!\!\int _ {0}^{\vartheta_\gamma(\varphi)}
{\mathcal G}(\varphi, \vartheta ; \alpha, \gamma) \cos ^{\beta-n} \vartheta\sin^{n-2}\vartheta d\vartheta\nonumber\\
& &\nonumber\\
& &-\int _ {0}^{\pi}\sin ^{n-3}\varphi d\varphi \!\!\int ^{\pi/2}_{\vartheta_\gamma(\varphi)}
{\mathcal G}(\varphi, \vartheta ; \alpha, \gamma) \cos ^{\beta-n} \vartheta\sin^{n-2}\vartheta d\vartheta\;.
\end{eqnarray}
On the other hand, by (\ref{eqn_001}),
\begin{eqnarray} \label{EH_inf4_02}
c_{n,\beta}(\alpha)&\!\!\!\!=\!\!\!\!&\int _ {0}^{\pi}\sin ^{n-3}\varphi d\varphi \!\!\int _ {0}^{\vartheta_\gamma(\varphi)}
{\mathcal G}(\varphi, \vartheta ; \alpha, \gamma) \cos ^{\beta-n} \vartheta\sin^{n-2}\vartheta d\vartheta\nonumber\\
& &\nonumber\\
&\!\!\!\!+\!\!\!\!&\int _ {0}^{\pi}\sin ^{n-3}\varphi d\varphi \!\!\int ^{\pi/2}_{\vartheta_\gamma(\varphi)}
{\mathcal G}(\varphi, \vartheta ; \alpha, \gamma) \cos^{\beta-n} \vartheta\sin^{n-2}\vartheta d\vartheta\;.
\end{eqnarray}
 
Using the equalities (\ref{EH_inf4_01}) and (\ref{EH_inf4_02}), and taking into account (\ref{eqn_002}), 
we rewrite (\ref{EH_inf1}) as
\begin{eqnarray} \label{EH_inf4}
C_{\alpha, \beta, \infty}\!\!\!\!&=&\!\!\!\!|k|\beta\sup _{\gamma \geq 0}\frac{ \omega _{n-2}}{\sqrt{1+\gamma^2}}\left \{ -c_{n, \beta}(\alpha)+2
\int _ {0}^{\pi}\sin ^{n-3}\varphi\;d\varphi \int _ {0}^{\vartheta_\gamma(\varphi )}
{\mathcal G}(\varphi, \vartheta ; \alpha, \gamma)\cos ^{\beta-n} \vartheta\sin^{n-2}\vartheta d\vartheta \right \}\nonumber\\
& &\nonumber\\
&=&|k|\beta\sup _{\gamma \geq 0}\frac{\omega _{n-2}}{\sqrt{1+\gamma^2}}\!\left \{ -c_{n,\beta}(\alpha)+2\int _ {0}^{\pi}\!\!
H_n\big ( \varphi, \vartheta _\gamma(\varphi); \alpha, \gamma  \big ) \sin ^{n-3}\varphi d\varphi \right \}.
\end{eqnarray}

By (\ref{EH_inf3}),
\begin{equation} \label{EH_sin}
\sin \vartheta _\gamma(\varphi)=\frac{h_\gamma(\varphi)}{\sqrt{4\alpha^2+h_\gamma^2(\varphi)}}\; ,
\end{equation}
\begin{equation} \label{EH_cos}
\cos \vartheta _\gamma(\varphi)=\frac{2\alpha}{\sqrt{4\alpha^2+h_\gamma^2(\varphi)} }\;,
\end{equation}
where $h_\gamma(\varphi)$ is defined by (\ref{EQN_1}).

Using (\ref{EH_sin}) and (\ref{EH_cos}), we find
\begin{equation} \label{EH_comb}
\sin ^{n-1}\vartheta_\gamma(\varphi )\cos ^{\beta-n+1}\vartheta_\gamma(\varphi )=\frac{(2\alpha)^{\beta-n+1} h^{n-1}_\gamma(\varphi )}{(4\alpha^2+h^2_\gamma(\varphi ))^{\beta/2}}\;.
\end{equation}
By (\ref{eqn_002}) and (\ref{EH_comb}) we can write $H_n\big ( \varphi, \vartheta _\gamma(\varphi); \alpha, \gamma  \big )$ as
\begin{eqnarray*} 
H_n\big ( \varphi, \vartheta _\gamma(\varphi); \alpha, \gamma  \big )\!\!\!\!&=&\!\!\!\!\frac{(2\alpha)^{\beta-n+1} h^{n-1}_\gamma(\varphi )}{\beta(4\alpha^2+h^2_\gamma(\varphi ))^{\beta/2}}\\
& &\nonumber\\
\!\!\!\!&+&\!\!\!\!\int _ {0}^{\arctan \frac{h_\gamma(\varphi)}{2\alpha}}\left \{ \!\left (\frac{\beta\!-\!n\!+\!1}
{\beta}\!-\!\alpha \right )\!+\!\gamma\cos \varphi\cos \vartheta\sin\vartheta
\!\right \}\!\cos ^{\beta\!-\!n}\vartheta\sin^{n\!-\!2}\vartheta d\vartheta \;,
\end{eqnarray*}
which together with (\ref{EH_inf4}) leads to
\begin{equation} \label{EA_idn}
C_{\alpha, \beta, \infty}=|k|\beta\sup _{\gamma \geq 0}\frac{ \omega _{n-2}}{\sqrt{1+\gamma^2}}\left \{ -c_{n, \beta}(\alpha)+2\int _ {0}^{\pi} 
P\big ( h_\gamma(\varphi )\big )\sin ^{n-3}\varphi d\varphi\right \} ,
\end{equation}
where
\begin{equation} \label{EC_idn}
P(z)=\frac{(2\alpha)^{\beta-n+1}z^{n-1}}{\beta(4\alpha^2 +z^2)^{\beta/2}}+\int _ {0}^{\arctan \frac{z}{2\alpha}} 
\left \{ \!\left (\frac{\beta\!-\!n\!+\!1}{\beta}\!-\!\alpha \right )\!+\!\gamma\cos \varphi\cos \vartheta\sin\vartheta
\!\right \}\!\cos ^{\beta\!-\!n}\vartheta\sin^{n\!-\!2}\vartheta d\vartheta\;.
\end{equation}
Equalities (\ref{EA_idn}), (\ref{EC_idn}) together with (\ref{eqn_001}), (\ref{EQN_1}) prove the Lemma.
\end{proof}

In the next assertion we consider a particular case of (\ref{EA_idnX}) for $\alpha=1$, $\beta \in (n-1, n]$. 
To find the explicit formula for $C_{1, \beta, \infty}$ we solve an extremal problem 
with a scalar parameter in the integrand of a double integral.

\begin{theorem} \label{SP_2} Let $\alpha =1$ and $\beta \in (n-1, n]$. Let $f\in L^\infty({\mathbb R}^{n-1})$,
and let  $x $ be  an arbitrary point in ${\mathbb R}^n _+$.
The sharp coefficient ${\mathcal C}_{1, \beta, \infty} (x)$ in the inequality
\begin{equation} \label{EH_X}
|\nabla u(x)|\leq {\mathcal C}_{1, \beta, \infty}( x)\big|\!\big |f \big |\!\big |_\infty
\end{equation}
is given by
\begin{equation} \label{EH_Y}
{\mathcal C}_{1, \beta, \infty}( x)=\frac{|k|\pi^{(n-1)/2}(n-1)
\Gamma\left (\frac{\beta-n+1}{2} \right ) } {\Gamma\left (\frac{\beta}{2} \right )} x_n^{n-2}\;.
\end{equation}

The absolute value of the derivative of $u$
with respect to the normal to the boundary of the half-space at any 
$x \in {\mathbb R}^n_+$ has the same supremum as $|\nabla u(x)|$.
\end{theorem}
\begin{proof} The inequality (\ref{EH_X}) follows from (\ref{EH_XAAB}).
By (\ref{EC_jjj}), in the case $\alpha=1$ we have $h_\gamma(\varphi)=2\gamma \cos \varphi $ for 
$\varphi\in [0, \pi/2]$ and $h_\gamma(\varphi)=0$ for $\varphi\in (\pi/2, \pi ]$.
Therefore, by (\ref{EA_idnX})-(\ref{EC_iii}) we obtain
\begin{equation} \label{eqn_009}
C_{1, \beta, \infty}=|k|\beta\sup _{\gamma \geq 0}\frac{\omega _{n-2}}{\sqrt{1+\gamma^2}}\!\left \{ -c_{n,\beta}(1)+2\int _0^{\pi/2}\!\!
U(\gamma\cos\varphi )\sin ^{n-3}\varphi d\varphi \right \}\;,
\end{equation}
where
\begin{equation} \label{eqn_010A}
c_{n, \beta}(1)=\frac{\sqrt{\pi}(1-n)\Gamma\left (\frac{n-2}{2} \right )
\Gamma\left (\frac{\beta-n+1}{2} \right )}{2\beta\Gamma\left (\frac{\beta}{2} \right )}
\end{equation}
and
\begin{equation} \label{eqn_010}
U(z)= \frac{z^{n-1}}{ \beta \big ( 1+z^2 \big )^{\beta/2}} +
\int _0^{\arctan z} \left ( \frac{1-n}{\beta}+z\cos \vartheta\sin \vartheta\right )
\cos^{\beta-n}\vartheta\sin^{n-2}\vartheta d\vartheta \;.
\end{equation}

Denoting
\begin{equation} \label{eqn_011}
F(\gamma)=\frac{1}{\sqrt{1+\gamma^2}}\!\left \{ -c_{n,\beta}(1)+2\int _ {0}^{\pi/2}\!\!
U(\gamma\cos\varphi )\sin ^{n-3}\varphi d\varphi \right \}\;,
\end{equation} 
where $U(z)$ is defined by (\ref{eqn_010}), we can rewrite (\ref{eqn_009}) in the form
\begin{equation} \label{eqn_009A}
C_{1, \beta, \infty}=|k|\beta\omega_{n-2}\sup _{\gamma \geq 0}F(\gamma)\;.
\end{equation}

It follows from (\ref{eqn_011}) that
\begin{equation} \label{eqn_012}
\frac{dF}{d\gamma}=\frac{1}{(1+\gamma^2)^{3/2}}\left \{ c_{n, \beta}(1)\gamma +2\int _ {0}^{\pi/2}\!\!
\left (-\gamma U(\gamma\cos\varphi ) +(1+\gamma^2)\frac{\partial U(\gamma\cos\varphi )}{\partial \gamma }\right )\sin ^{n-3}\varphi d\varphi \right \}\;.
\end{equation}
Differentiating $U(\gamma\cos\varphi )$ with respect to $\gamma$, we obtain
\begin{equation} \label{eqn_013}
\frac{\partial U(\gamma\cos \varphi )}{\partial\gamma}=\cos \varphi\int _0^{\arctan (\gamma \cos \varphi)}\cos^{\beta-n+1}\vartheta \sin ^{n-1}\vartheta d\vartheta\;.
\end{equation}
Substituting $U(\gamma\cos \varphi )$ from (\ref{eqn_010}) and $\partial U(\gamma\cos \varphi )/\partial\gamma$ 
from (\ref{eqn_013}) into (\ref{eqn_012}), we arrive at equality
\begin{equation} \label{eqn_014}
\frac{dF}{d\gamma}=\frac{1}{(1+\gamma^2)^{3/2}}\Big \{ \Phi_1(\gamma )+ \Phi_2(\gamma )\Big \}\;,
\end{equation}
where
$$
\Phi_1(\gamma )=\frac{\gamma(n-1)}{\beta}\left \{-\frac{\sqrt{\pi}\Gamma\left (\frac{n-2}{2} \right )
\Gamma\left (\frac{\beta-n+1}{2} \right )}{2\Gamma\left (\frac{\beta}{2} \right )} +2\int_0^{\pi/2}
\left ( \int _ 0^{\arctan (\gamma \cos\varphi)} \cos^{\beta-n}\vartheta\sin^{n-2}\vartheta d\vartheta \right )\sin ^{n-3}\varphi d\varphi \right \}
$$
and
$$
\Phi_2(\gamma )=2\left\{- \frac{\gamma}{\beta}\int _ {0}^{\pi/2}\!\!\frac{(\gamma \cos\varphi)^{n-1}\sin ^{n-3}\varphi}
{\big ( 1+\gamma^2 \cos^2\varphi \big )^{\beta/2}}d\varphi+
\int _ {0}^{\pi/2}\!\!\left ( \int _ {0}^{\arctan (\gamma \cos\varphi)}\!\!\!\!\!\!\!\!\!\!\!\!\!\!\!\!\!\!\!\!\!\!\! \cos^{\beta-n+1}\vartheta\sin^{n-1}\vartheta d\vartheta\right )\cos\varphi\sin ^{n-3}\varphi d\varphi \right \}\;.
$$
Estimating $\Phi_1(\gamma )$, we obtain
\begin{equation} \label{eqn_015}
\Phi_1(\gamma )\leq \frac{\gamma(n\!-\!1)}{\beta}\left \{-\frac{\sqrt{\pi}\Gamma\left (\frac{n\!-\!2}{2} \right )
\Gamma\left (\frac{\beta\!-\!n\!+\!1}{2} \right )}{2\Gamma\left (\frac{\beta}{2} \right )}\!+2\!\int_0^{\pi/2}\!\!
\left ( \int _ 0^{\pi/2}\! \cos^{\beta-n}\vartheta\sin^{n-2}\vartheta d\vartheta \right )\sin ^{n-3}\varphi d\varphi \right \}=0\;.
\end{equation} 
By differentiating $\Phi_2(\gamma )$, we arrive at equality
$$
\frac {d\Phi_2}{d\gamma }=2 \int _ {0}^{\pi/2}\frac{( \gamma\cos\varphi )^{n-1}}
{\big ( 1+\gamma^2 \cos^2\varphi \big )^{(\beta+2)/2}}\left \{ -\frac{n}{\beta}+\left (1-\frac{n}{\beta}\right )\gamma^2\cos^2 \varphi 
+\cos\varphi \right \}\sin^{n-3}\varphi d\varphi\;.
$$
Therefore, $\Phi '_2(\gamma)<0$ for $\gamma >0$ and any $\beta \in (n-1, n]$.
This together with $\Phi_2(0)=0$ proves inequality $\Phi_2(\gamma)<0$ for $\gamma >0$ and any $\beta \in (n-1, n]$. 
Hence, by (\ref{eqn_014}) and (\ref{eqn_015}), $F'(\gamma)<0$ for $\gamma >0$ and any $\beta \in (n-1, n]$. So, by (\ref{eqn_009A}),
$$
C_{1, \beta, \infty}=|k|\beta\omega_{n-2}F(0)\;,
$$
which, in view of (\ref{eqn_010A})-(\ref{eqn_011}), leads to 
$$
C_{1, \beta, \infty}=-|k|\beta\omega _{n-2}c_{n, \beta}(1)=
|k|\frac{\pi^{(n-1)/2}(n-1)\Gamma\left (\frac{\beta-n+1}{2} \right )}{\Gamma\left (\frac{\beta}{2} \right )}\;.
$$
Combining the last formula with (\ref{EH_XAAC}) in the case $\alpha=1$, we arrive at (\ref{EH_Y}).
Since $\gamma =|\bs z'|/z_n$ and the supremum in (\ref{eqn_009A}) is attained at $\gamma=0$, we conclude that the coefficient ${\mathcal C}_{1, \beta, \infty}(x)=C_{1, \beta, \infty}x_n^{n-2}$ is sharp under conditions of the Theorem also in the weaker inequality, 
obtained from $(\ref{EH_X})$ by replacing $\nabla u$ by $\partial u/\partial x_n$.
\end{proof}

%%%%%%%%%%%%%%%%%%%%%%%%%%%%%%%%%%%%%%%%%%%%%%%%%%%%%%%%%%%%%%%%%
\section{Sharp estimates for harmonic and biharmonic functions}
%%%%%%%%%%%%%%%%%%%%%%%%%%%%%%%%%%%%%%%%%%%%%%%%%%%%%%%%%%%%%%%%%

By $h^p({\mathbb R}^n_+)$ we denote the Hardy space of harmonic functions on ${\mathbb R}^n_+$ which can be represented as the Poisson integral
\begin{equation} \label{Eq_8.1}
v(x)=\frac{2}{\omega_n}\int_{{\mathbb R}^{n-1}}\frac{x_n}{|y-x|^n}v(y')dy'
\end{equation}
with boundary values in $L^p({\mathbb R}^{n-1})$, $1\leq p \leq \infty$. Multiplying (\ref{Eq_8.1}) on $x_n^{n\alpha - 1}$, 
$\alpha\geq 0$, we obtain 
\begin{equation} \label{Eq_8.2A}
x_n^{n\alpha - 1}v(x)=\frac{2}{\omega_n}\int_{{\mathbb R}^{n-1}}\left ( \frac{x_n^\alpha}{|y-x|} \right )^n v(y')dy'.
\end{equation}
On the right-hand side of the last equality is located generalized Poisson integral (\ref{EI_1}) with $k=2/\omega_n$ and $\beta=n$. 

Thus, we can apply the results of previous sections to obtain sharp pointwise estimates for 
$\left |\nabla \big ( x_n^{n\alpha - 1}v(x)\big )\right |$ in terms of the norm $L^p({\mathbb R}^{n-1}), 1\leq p \leq \infty$.

As consequence of Proposition \ref{P_1} and Theorem \ref{P_1_0} with $\beta=n$, we obtain

\begin{corollary} \label{C_8.1}
Let $v \in h^p({\mathbb R}^n_+)$ and let $x $ be  an arbitrary point in ${\mathbb R}^{n} _+$. % and let $\bs z\in {\mathbb S}^{n-1}$. 
The sharp coefficient ${\mathcal C}_{\alpha, n, p} (x)$ in the inequality
$$
\left |\nabla \big ( x_n^{n\alpha - 1}v(x)\big )\right |\leq {\mathcal C}_{\alpha, n, p}( x)\big|\!\big |v \big |\!\big |_p
$$
is given by
$$
{\mathcal C}_{\alpha, n, p}( x)= C_{\alpha, n, p}\; x_{n}^{n\alpha -2-(n-1)/p}\;,
$$
where
$$
C_{\alpha, n, 1}=\frac{2 n}{\omega_n}\sup_{|\bs z|=1}\sup _{\sigma \in {\mathbb S}^{n-1}_+}
\big |\big (\alpha\bs e_{n} -(\bs e_{\sigma}, \bs e_{n})\bs e_{\sigma},\; \bs z \big )\big |
\big (\bs e_{\sigma}, \bs e_{n} \big )^n ,
$$
$$
C_{\alpha, n, p}\!=\!\frac{2 n}{\omega_n}\sup_{|\bs z|=1}
\left \{ \int _ {{\mathbb S}^{n-1}_+ }\!\!
\big |\!\big (\alpha\bs e_{n} \!-\!(\bs e_{\sigma}, \bs e_{n})\bs e_{\sigma}, \bs z \big )\!\big |^{\frac{p}{p\!-\!1}}
\big (\bs e_{\sigma}, \bs e_{n} \big )^{\frac{n}{p-1}}d\sigma \!\right \}^{\frac{p\!-\!1}{p}}
$$
for $1<p<\infty $,
and
$$
C_{\alpha, n, \infty}=\frac{2 n}{\omega_n}\sup_{|\bs z|=1}\int _ {{\mathbb S}^{n-1}_+ }
\big |\big (\alpha\bs e_{n} - (\bs e_{\sigma}, \bs e_{n})\bs e_{\sigma},\; 
\bs z \big )\big |\;d\sigma .
$$

In particular, the sharp constant $C_{0, n, p}$ in the inequality
\begin{equation} \label{Eq_8.7}
\left |\nabla \left \{ \frac{v (x)}{x_n} \right \}\right |
\leq C_{0, n, p} x_{n}^{-2-(n-1)/p}\;\big|\!\big |v \big |\!\big |_p
\end{equation}
is given by $C_{0, n, 1}=2n/\omega_n$, $C_{0, n, \infty}=1$ and
$$
C_{0, n, p}=\frac{2 n}{\omega_n}\left \{\frac{\pi^{\frac{n-1}{2}}\Gamma\left ( \frac{3p+n-1}{2(p-1)} \right )}
{\Gamma\left ( \frac{(n +2)p}{2(p-1)} \right )} \right \}^{\frac{p}{p-1}}
$$
for $1< p< \infty$.

The constant $C_{0, n, p}$ is sharp in conditions of the Corollary also in the weaker inequality 
obtained from $(\ref{Eq_8.7})$ by replacing $\nabla \big ( x_n^{n\alpha - 1}v\big )$ 
by $\partial \big ( x_n^{n\alpha - 1}v\big ) /\partial x_n$.
\end{corollary}

Concretizing Theorem \ref{C_1_1} for $\beta=n$, we arrive at

\begin{corollary} \label{C_8.2} Let $v \in h^1({\mathbb R}^n_+)$
and let $x $ be  an arbitrary point in ${\mathbb R}^{n} _+$. 
The sharp coefficient ${\mathcal C}_{\alpha, n, 1} (x)$ in the inequality
\begin{equation} \label{Eq_8.9}
\left |\nabla \big ( x_n^{n\alpha - 1}v(x)\big )\right |\leq {\mathcal C}_{\alpha, n, 1}( x)\big|\!\big |v \big |\!\big |_1
\end{equation}
is given by
$$
{\mathcal C}_{\alpha, n, 1}(x)=C_{\alpha, n, 1}\;x_{n}^{n(\alpha-1) -1}\;,
$$
where
$$
C_{\alpha, n, 1}=\frac{2n}{\omega_n}|1-\alpha|
$$
if 
\begin{equation} \label{Eq_8.12}
0\leq\alpha\leq\frac{\sqrt{n+1}}{\sqrt{n+1}+1}\;\;\;{\rm or}\;\;\;\alpha\geq\frac{\sqrt{n+1}}{\sqrt{n+1}-1}\;,
\end{equation}
and
$$
C_{\alpha, n, 1}=\frac{2n}{\omega_n}\left ( \frac{n }{2\alpha-1} \right )^{n/2}\left ( \frac{\alpha^2}{n+1} \right )^{(n+2)/2} 
$$
if
$$
\frac{\sqrt{n+1}}{\sqrt{n+1}+1} < \alpha< \frac{\sqrt{n+1}}{\sqrt{n+1}-1}\;.
$$

In particular,
$$
C_{1, n, 1}=\frac{2(n-2)}{n\omega_n}\left \{ \frac{(n-1)^2}{(n-2)(n+1)} \right \}^{(n+2)/2}\;.
$$

If $\alpha$ satisfies condition $(\ref{Eq_8.12})$, then the coefficient ${\mathcal C}_{\alpha, n, 1}( x)$ is sharp 
in conditions of the Corollary also in
the weaker inequality obtained from $(\ref{Eq_8.9})$ by replacing $\nabla \big ( x_n^{n\alpha - 1}v\big )$ 
by $\partial \big ( x_n^{n\alpha - 1}v\big ) /\partial x_n$.
\end{corollary}

Theorem \ref{C_2_2} in the case $\beta=n$ leads to the following
 
\begin{corollary} \label{C_8.3} Let $v \in h^2({\mathbb R}^n_+)$
and let  $x $ be  an arbitrary point in ${\mathbb R}^{n} _+$. 
The sharp coefficient ${\mathcal C}_{\alpha, n, 2} (x)$ in the inequality
\begin{equation} \label{Eq_8.15}
\left |\nabla \big ( x_n^{n\alpha - 1}v(x)\big )\right |\leq {\mathcal C}_{\alpha, n, 2}( x)\big|\!\big |v \big |\!\big |_2
\end{equation}
is given by
$$
{\mathcal C}_{\alpha, n, 2}( x)= C_{\alpha, n, 2}\;x_{n}^{n\alpha -(n+3)/2}\;,
$$
where
$$
C_{\alpha, n, 2}=\left \{ \frac{n}{2^{n-2}\omega_n}
\left (  n\alpha^2  -(n+1)\alpha+ \frac{n+3}{4}\right )\right \}^{1/2}
$$
if
\begin{equation} \label{Eq_8.18}
0\leq\alpha\leq\frac{1}{2}\;\;\;{\rm or}\;\;\;\alpha\geq\frac{1}{2}+\frac{1}{n}\;,
\end{equation}
and
$$
C_{\alpha, n, 2}=\left \{ \frac{n}{2^{n}\omega_n }\right \}^{1/2}
$$
if 
$$
\frac{1}{2}<\alpha <\frac{1}{2}+\frac{1}{n}\;.
$$

In particular,
$$
C_{1, n, 2}=\left \{ \frac{n(n-1)}{2^{n}\omega_n }\right \}^{1/2}\;.
$$

If $\alpha$ satisfies condition $(\ref{Eq_8.18})$, 
then the coefficient ${\mathcal C}_{\alpha, n, 2}( x)$ is sharp in conditions of the Corollary also in
the weaker inequality obtained from $(\ref{Eq_8.15})$ by replacing $\nabla \big ( x_n^{n\alpha - 1}v)\big )$ by
$\partial \big ( x_n^{n\alpha - 1}v)\big )/\partial x_n$.
\end{corollary}

As consequence of Theorem \ref{SP_2AAA} with $\beta=n$ we obtain

\begin{corollary} \label{C_8.4} Let $v \in h^\infty({\mathbb R}^n_+)$
and let  $x $ be  an arbitrary point in ${\mathbb R}^n _+$. Let $\alpha_n$ be the root from the interval $(1, +\infty )$ 
of the equation  
$$
\frac{2}{\pi(n\alpha -1 )}=\frac{\alpha-1}{1+\sqrt{1+(\alpha-1)^2}}
$$
with respect to $\alpha$. 

If $\alpha \geq \alpha_n$, then the sharp coefficient ${\mathcal C}_{\alpha, n, \infty} (x)$ in the inequality
$$
\left |\nabla \big ( x_n^{n\alpha - 1}v(x)\big )\right |\leq {\mathcal C}_{\alpha, n, \infty}( x)\big|\!\big |v \big |\!\big |_\infty
$$
is given by
$$
{\mathcal C}_{\alpha, n, \infty}( x)=(n\alpha - 1)\;x_{n}^{n\alpha -2}\;.
$$

In conditions of the Corollary, absolute value of the derivative of $ x_n^{n\alpha - 1}v$
with respect to the normal to the boundary of the half-space at any
$x \in {\mathbb R}^n_+$ has the same supremum as $|\nabla \big ( x_n^{n\alpha - 1}v\big )|$.
\end{corollary}

Theorem \ref{SP_2} with $\beta=n$ implies

\begin{corollary} \label{C_8.5} Let $v \in h^\infty({\mathbb R}^n_+)$
and let  $x $ be  an arbitrary point in ${\mathbb R}^n _+$.
The sharp coefficient ${\mathcal C}_{1, n, \infty}(x)$ in the inequality
$$
\left |\nabla \big ( x_n^{n - 1}v(x)\big )\right |\leq {\mathcal C}_{1, n, \infty}( x)\big|\!\big |v \big |\!\big |_\infty
$$
is given by
$$
{\mathcal C}_{1, n, \infty}( x)=(n-1) x_n^{n-2}\;.
$$

The absolute value of the derivative of $ x_n^{n - 1}v(x)$ with respect to the normal to the boundary of 
the half-space at any $x \in {\mathbb R}^n_+$ has the same supremum as $\left |\nabla \big ( x_n^{n - 1}v(x)\big )\right |$.
\end{corollary}

\medskip
We conclude this section by some remark. The following representation is well known (e.g. Schot \cite{SCH}) 
\begin{equation} \label{EQ_4IA}
w(x)=\frac{2 n}{\omega_n}\int_{{\mathbb R}^{n-1}}\frac{x^3_n}{|y-x|^{n+2}}f_0(y')dy'+ 
\frac{2}{\omega_n}\int_{{\mathbb R}^{n-1}} \frac{x^2_n}{|y-x|^{n}}f_1(y')dy'
\end{equation}
for solution in ${\mathbb R}^n_+$ of the first boundary value problem for biharmonic equation
\begin{equation} \label{EQ_3IA}
\Delta ^2 w=0\;\;in\;\;{\mathbb R}^n_+\;,\;\;\;w\big |_{x_n=0}=f_0(x'),\;\;\;\frac{\partial w}{\partial x_n}\Big |_{x_n=0}=f_1(x')\;,
\end{equation}
where $y=(y', 0)$.

By $w_0$ we denote a solution of the problem (\ref{EQ_3IA}) with $f_0=0$. By (\ref{EQ_4IA}), we have
\begin{equation} \label{EQ_4IAAA}
x_n^{n\alpha - 2} w_0(x)=\frac{2}{\omega_n}\int_{{\mathbb R}^{n-1}} \frac{x_n^{n\alpha}}{|y-x|^{n}}f_1(y')dy'\;.
\end{equation}
The right-hand side of (\ref{EQ_4IAAA}) is the same as in (\ref{Eq_8.2A}). So, by Proposition \ref{P_1} we arrive at
\begin{equation} \label{EI_8FFF}
\left |\nabla \big ( x_n^{n\alpha - 2}w_0(x)\big )\right |\leq C_{\alpha, n, p}x_n^{n\alpha-2-(n-1)/p}\left |\!\left |
\frac{\partial w_0}{\partial x_n} \right |\!\right |_p\;,
\end{equation}
where the sharp constant $C_{\alpha, n, p}$ in (\ref{EI_8FFF}) is the same as in Corollaries \ref{C_8.1}-\ref{C_8.5}. 

For instance, in the case $\alpha=0, p=\infty$ by Corollary \ref{C_8.1} we have the following inequality with the sharp coefficient
$$
\left |\nabla \left \{ \frac{w_0 (x)}{x_n^2} \right \}\right |
\leq \frac{1}{x_n^2} \;\left |\!\left |
\frac{\partial w_0}{\partial x_n} \right |\!\right |_\infty\;.
$$
In other case, for $\alpha=1$ and $p=\infty$, Corollary \ref{C_8.5} leads to the next inequality
$$
\left |\nabla \big ( x_n^{n - 2}w_0(x)\big )\right |\leq (n-1) x_n^{n-2}\left |\!\left |
\frac{\partial w_0}{\partial x_n} \right |\!\right |_\infty\;.
$$

%%%%%%%%%%%%%%%%%%%%%%%%%%%%%%%%%%%%%%%%%%%%%%%%%%%%%%%%%%%%%%%%%

\end{document}